\documentclass{article}

\usepackage{geometry}

\usepackage[english]{babel}
\usepackage{amsmath,amsthm,amssymb}
\usepackage{tikz-cd}
\usepackage{dsfont}
\usepackage{picture,float,graphicx,listings}
\usepackage{hyperref,verbatim}

\usepackage{todonotes}

\newtheorem{theorem}{Theorem}[section]
\newtheorem{lemma}[theorem]{Lemma}
\newtheorem{corollary}[theorem]{Corollary}
\newtheorem{proposition}[theorem]{Proposition}
\theoremstyle{remark}
\newtheorem*{remark}{Remark}

\theoremstyle{definition}
\newtheorem{definition}[theorem]{Definition}

\newtheorem{theoremIntro}{Theorem}
\newtheorem{propositionIntro}[theoremIntro]{Proposition}

\newcommand{\Z}{\mathbb{Z}}

\newcommand{\R}{\mathbb{R}}
\newcommand{\C}{\mathbb{C}}

\renewcommand{\S}{\mathbb{S}}

\newcommand{\tend}[2]{\underset{#1 \rightarrow #2}{\longrightarrow}}
\newcommand{\landau}[2]{\underset{#1 \rightarrow #2}{=}}

\newcommand{\rk}{\mathrm{rk}}

\newcommand{\Hom}{\mathrm{Hom}}
\newcommand{\End}{\mathrm{End}}

\newcommand{\Aut}{\mathrm{Aut}}

\newcommand{\tr}{\mathrm{tr}}
\renewcommand{\det}{\mathrm{det}}

\newcommand{\Vol}{\mathrm{Vol}}

\newcommand{\Lie}{\mathrm{Lie}}

\newcommand{\Id}{\mathrm{Id}}

\renewcommand{\i}{\boldsymbol{\mathrm{i}}}
\newcommand{\e}{\boldsymbol{\mathrm{e}}}
\newcommand{\scal}[2]{\left< #1,#2 \right>}

\newcommand{\norme}[1]{\left\| #1 \right\|}

\newcommand{\E}{\mathcal{E}}
\newcommand{\F}{\mathcal{F}}
\newcommand{\G}{\mathcal{G}}
\newcommand{\Gr}{\mathrm{Gr}}

\newcommand{\dbar}{\overline{\partial}}
\newcommand{\mC}{\mathcal{C}}

\renewcommand{\leq}{\leqslant}
\renewcommand{\geq}{\geqslant}
\newcommand{\inter}[1]{#1^o}

\newcommand{\fonction}[4]{\left\{ \begin{array}{rcl}
		\displaystyle #1 & \longrightarrow & \displaystyle #2\\
		\displaystyle #3 & \longmapsto & \displaystyle #4
	\end{array} \right.}

\title{Continuity of HYM Connections \\ with Respect to Metric Variations}
\author{Rémi Delloque}
\date{}

\begin{document}

\maketitle
\begin{abstract}
    We investigate the set of (real Dolbeault classes of) balanced metrics $\Theta$ on a balanced manifold $X$ with respect to which a torsion-free coherent sheaf $\E$ on $X$ is slope stable. We prove that the set of all such $[\Theta] \in H^{n - 1,n - 1}(X,\R)$ is an open convex cone locally defined by a finite number of linear inequalities.
    
    When $\E$ is a Hermitian vector bundle, the Kobayashi--Hitchin correspondence provides associated Hermitian Yang--Mills connections, which we show depend continuously on the metric, even around classes with respect to which $\mathcal{E}$ is only semi-stable. In this case, the holomorphic structure induced by the connection is the holomorphic structure of the associated graded object. The method relies on semi-stable perturbation techniques for geometric PDEs with a moment map interpretation and is quite versatile, and we hope that it can be used in other similar problems.
\end{abstract}

\section{Introduction}

The famous Kobayashi--Hitchin correspondence \cite{KH1,KH2,Donaldson,Uhlenbeck_Yau} gives the equivalence between the existence of Hermitian Yang--Mills (HYM) connections on a Hermitian holomorphic vector bundle $\E \rightarrow X$ and slope stability \cite{Mumford,Takemoto}. HYM connections correspond to solutions of a partial differential equation while slope stability is a purely algebraic notion. These tools are used to construct and study the moduli space of vector bundles on a given complex manifold.

Both of these notions depend on the balanced metric on the compact complex manifold $X$ (balanced metrics are defined below). It is a natural question to ask how $\E$ may become stable or unstable when this metric varies, and to study the behaviour of the associated HYM connections. It is easy to see that stability is an open condition in the space of balanced metrics and that we can locally build a smooth family of associated HYM connections around a metric $\Theta_0$ with respect to which $\E$ is stable thanks to the implicit function theorem. However, when $\E$ is only semi-stable with respect to $\Theta_0$, problems occur and wall-crossing phenomena appear. The method used here is inspired from \cite{Gabor,DMS,Clarke_Tipler}, and relies on semi-stable perturbation techniques for geometric PDEs with a moment map interpretation.

Let $g$ be a Hermitian metric on $X$ and $\omega$ be the associated $(1,1)$-form. Then, by definition, $(X,\omega)$ is Kähler if and only if $d\omega = 0$. $\E$ is said to be slope stable with respect to $[\omega]$ if for all coherent sub-sheaves $\mathcal{S}$ of $\E$ with $0 < \rk(\mathcal{S}) < \rk(\E)$,
$$
\mu_{[\omega^{n - 1}]}(\mathcal{S}) < \mu_{[\omega^{n - 1}]}(\E).
$$
Here, $n = \dim(X)$ and,
$$
\mu_{[\omega^{n - 1}]}(\E) = \frac{c_1(\E) \cup [\omega^{n - 1}]}{\rk(\E)},
$$
is the slope of $\E$. We see that these inequalities are linear in $\omega^{n - 1}$ and $\omega$ is not required to be closed, only $\omega^{n - 1}$ is. When $\omega^{n - 1}$ is closed (but not necessarily $\omega$), we say that $X$ is balanced and this is enough to get the Kobayashi--Hitchin correspondence. This is inspired from the idea of Greb, Ross and Toma in \cite{GRT} where wall-crossing phenomena are studied with respect to multipolarisations to make the equations linear instead of polynomial. We recall general definitions and results about balanced manifolds in Section \ref{SEC:Généralités}.

From here, the inequalities that define slope stability are linear in $\Theta = \frac{\omega^{n - 1}}{(n - 1)!}$. As in the Kähler case, we define the balanced cone $\mathcal{B}_X$ as the set of all real Dolbeault classes,
$$
[\Theta] \in H^{n - 1,n - 1}(X,\R) = H^{n - 1,n - 1}(X,\C) \cap H^{2n - 2}(X,\R),
$$
of positive $(n - 1,n - 1)$ forms $\Theta$. In Section \ref{SEC:Finitude des c_1}, we prove that locally, there is only a finite number of these inequalities to verify in order to obtain all of them. This is a consequence of the following result (Corollary \ref{COR:Finitude locale} below).

\begin{propositionIntro}\label{PRO:1}
    Let $X$ be a compact complex balanced manifold of balanced cone $\mathcal{B}_X$ and $\E \rightarrow X$ a torsion-free coherent sheaf. Then, for all compact $K \subset \mathcal{B}_X$, there is a finite family $\mathcal{S}_1,\ldots,\mathcal{S}_p \subset \E$ with for all $k$, $0 < \rk(\mathcal{S}_k) < \rk(\E)$ such that for all $[\Theta] \in K$, $\E$ is $[\Theta]$-stable (resp. $[\Theta]$-semi-stable) if and only if for all $k$, $\mu_{[\Theta]}(\mathcal{S}_k) < \mu_{[\Theta]}(\E)$ (resp. $\mu_{[\Theta]}(\mathcal{S}_k) \leq \mu_{[\Theta]}(\E)$).
\end{propositionIntro}

In the smooth projective setting, stronger results of the same kind are already known, such as \cite[Theorem 6.7]{Greb_Toma}. It implies immediately that slope stability is an open condition, which can also be proven by analytic methods using the implicit function theorem \cite[Theorem 5.1.4]{Lübke_Teleman_KH}. Proposition \ref{PRO:1} gives the structure of the set of metrics with respect to which $\E$ is stable (resp. semi-stable).

\begin{theoremIntro}\label{THE:2}
    The set $\mC_s^{\mathrm{D}}(\E)$ (resp. $\mC_{ss}^{\mathrm{D}}(\E)$) of balanced classes of metrics $[\Theta]$ with respect to which $\E$ is stable (resp. semi-stable) is locally an open (resp. closed) convex polyhedral cone.
\end{theoremIntro}

A more precise version of this theorem with topological properties is given by Theorem \ref{THE:Structures cônes} below. Proposition \ref{PRO:1} and Theorem \ref{THE:2} are folklore and the main results of this paper are the analytical results which concern the behaviour of the associated HYM connections on $\E$.

Let $(E,h)$ be a Hermitian vector bundle on $X$. For each complex structure $\dbar$ on $E$, there is a unique $h$-unitary connection $\nabla$ on $E$ such that $\nabla^{0,1} = \dbar$. It is the \textit{Chern connection} of $(\dbar,h)$.

Following \cite[Section 4.1]{Kobayashi}, we say that a Chern connection $\nabla$ is a $\Theta$-Hermitian Yang--Mills connection (or satisfies the Einstein condition) if the curvature form $F_\nabla = \nabla \circ \nabla$ satisfies,
$$
\i F_\nabla \wedge \Theta = c\,\Vol_\Theta\,\Id_E,
$$
for some constant $c$. Here, $\Vol_\Theta = \frac{\omega^n}{n!}$ is the volume form associated to $\Theta = \frac{\omega^{n - 1}}{(n - 1)!}$ It is an order two non-linear elliptic equation motivated by gauge theory in physics.

\begin{remark}\label{REM:Formule c}
When it exists, $c$ is entirely determined by the topology of $E$, the real Dolbeault class $[\Theta] \in H^{n - 1,n - 1}(X,\R)$ and the volume of $X$. Indeed, by Chern-Weil theory, we have,
$$
c_\Theta = \frac{2\pi c_1(E) \cup [\Theta]}{\rk(E)\Vol_\Theta(X)},
$$
by integrating over $X$ the trace of the HYM equation.
\end{remark}
Since the Hermitian metric $h$ on $E$ is fixed, $\nabla$ depends only on the complex structure $\dbar$ of $E$ as it is the Chern connection associated with $(\dbar,h)$. We write $F_{\dbar} = F_\nabla$ and we also say with a slight abuse that $\dbar$ is $\Theta$-HYM.

The gauge group $\G^\C$ of $\E$ is defined as the group of smooth sections of $\End(E)$ which are isomorphisms on each fibre. It acts on the set of complex structures $\dbar$,
$$
f \cdot \dbar = f \circ \dbar \circ f^{-1}.
$$
When $\E = (E,\dbar_E,h)$ is a Hermitian holomorphic vector bundle, we are interested in finding HYM complex structures $\dbar$ which lie in the same gauge group of $\dbar_E$. It is equivalent to saying that $(E,\dbar_E)$ and $(E,\dbar)$ are biholomorphic.

A powerful analytic tool to study stability is the Kobayashi--Hitchin correspondence, which was originally proven for Kähler metrics by Donaldson on surfaces \cite{Donaldson} and by Uhlenbeck and Yau in higher dimension \cite{Uhlenbeck_Yau} and extended to more general cases, including balanced metrics \cite{Li_Yau}, and on Higgs bundles \cite{Lübke_Teleman}.

\begin{propositionIntro}[Kobayashi--Hitchin correspondence]\label{PRO:Kobayashi--Hitchin}
    If $\E = (E,\dbar_E,h)$ is a Hermitian holomorphic vector bundle of positive rank, then it is $[\Theta]$-polystable if and only if it admits a $\Theta$-HYM operator $\dbar$ in the gauge orbit of $\dbar_E$. In this case, it is unique up to a unitary gauge transformation and the decomposition of $\E$ in a direct sum of stable bundles is always both $\dbar$-holomorphic and $h$-orthogonal, making its components $\Theta$-HYM.
\end{propositionIntro}
\begin{remark}
Here, we fix $h$ and we find $\dbar$ in the gauge orbit of $\dbar_E$ so the Chern connection associated to $(\dbar,h)$ is $\Theta$-HYM. It is also possible to fix $\dbar_E$ and to find a metric $k$ such that the Chern connection associated to $(\dbar_E,k)$ is $\Theta$-HYM. In this case, there is a map, $f : (E,\dbar_E,k) \rightarrow (E,\dbar,h)$ which is both a biholomorphism and an isometry.
\end{remark}

Our purpose from now on is to extend locally the algebraic results of Theorem \ref{THE:Structures cônes} to analytic results. Concretely, we want to show that when $\Theta \in \mC_s(\E)$ approaches a point $\Theta_0 \in \mC_{ss}(\E)$, we can build some $\Theta$-HYM $\dbar$ in the gauge orbit of $\dbar_E$ that approaches some $\Theta_0$-HYM $\dbar_0$, which is in the gauge orbit of the Dolbeault operator of $\Gr_{[\Theta_0]}(\E)$ (under the additional condition that $\E$ is sufficiently smooth with respect to $[\Theta_0]$).

In Section \ref{SEC:Résulats analytiques}, we investigate the behaviour of the HYM connections when $\Theta$ approaches a metric $\Theta_0$ with respect to which $\E$ is only semi-stable. If $\E$ is $[\Theta_0]$-semi-stable, we can build a $[\Theta_0]$-polystable coherent sheaf $\Gr_{[\Theta_0]}(\E)$ from a Jordan-Hölder filtration of $\E$ (see Proposition \ref{PRO:FJH}). If this sheaf is locally free, we say that $\E$ is \textit{sufficiently smooth}. In this case, the underlying smooth structure of the graded object is isomorphic to $E$ in a canonical way (thanks to the metric $h$) and it admits a $\Theta_0$-HYM $h$-unitary connection $\nabla_0 = \partial_0 + \dbar_0$.

Let $\norme{\cdot}_A$ be a $\mC^k$ norm, a Hölder norm or a Sobolev norm that dominates the $\mC^1$ norm on the space of smooth forms with values in any smooth complex bundle on $X$. We set,
$$
A' = \left\{
\begin{array}{lcl}
    \mC^{k - 1} & \textrm{if} & A = \mC^k,\\
    \mC^{k - 1,\alpha} & \textrm{if} & A = \mC^{k,\alpha},\\
    L_{k - 1}^p & \textrm{if} & A = L_k^p.
\end{array}
\right.
$$
It defines a norm $\norme{\cdot}_{A'}$ which dominates the $\mC^0$ norm on the space of smooth forms with values in any smooth complex bundle on $X$. Let $\norme{\cdot}$ be any norm on the finite dimensional space $H^{n - 1,n - 1}(X,\R)$.

For all real smooth $(n - 1,n - 1)$-form $\varepsilon \in \Omega^{n - 1,n - 1}(X,\R)$ whose $\mC^0$ norm is small enough, $\Theta_0 + \varepsilon$ is positive, thus is a balanced metric. When $\E$ and sufficiently smooth (with respect to $[\Theta_0]$), we show a convergence result when $\Theta \rightarrow \Theta_0$ (Theorem \ref{THE:Famille de connections HYM} below).

\begin{theoremIntro}\label{THE:3}
    Assume that there exists a positive $(n - 1,n - 1)$ form $\Theta'$ such that $\E$ is $[\Theta']$-stable (in particular, it must be simple). For all $\Theta_\varepsilon = \Theta_0 + \varepsilon$ $\mC^0$-close to $\Theta_0$ such that $\E$ is $[\Theta_\varepsilon]$-semi-stable, there is an associated $h$-unitary connection $\nabla_\varepsilon = \partial_\varepsilon + \dbar_\varepsilon$ on $E$ that has the three following properties,
    \begin{enumerate}
        \item $\nabla_\varepsilon$ is $\Theta_\varepsilon$-HYM.
        \item $(E,\dbar_\varepsilon)$ is biholomorphic to $\Gr_{[\Theta_\varepsilon]}(\E)$, which is locally free. In particular, if $\E$ is $[\Theta_\varepsilon]$-stable, then $(E,\dbar_\varepsilon) \cong \E$.
        \item The bound $\norme{\dbar_\varepsilon - \dbar_0}_A = \mathrm{O}\!\left(\norme{\varepsilon}_{A'} + \sqrt{\norme{[\varepsilon]}}\right)$ holds (the $\mathrm{O}$ depends on $A$). Therefore, the same bound holds for $\norme{\partial_\varepsilon - \partial_0}_A$ and $\norme{\nabla_\varepsilon - \nabla_0}_A$. In particular, $\nabla_\varepsilon \rightarrow \nabla_0$ for the $\mC^\infty$ topologies.
    \end{enumerate}
\end{theoremIntro}

The HYM connection, when it exists in some complex gauge orbit, is not unique but it is exactly one $\G$-orbit where $\G \subset \G^\C$ is the group of smooth sections $u$ of endomorphisms of $E$ which verify $uu^* = \Id_E$. The space of all $h$-unitary connections on $E$ modulo $\G$ is well-defined and Hausdorff for the $\mC^\infty$ topology \cite[Corollary 7.1.15]{Kobayashi}. A consequence of Theorem \ref{THE:3} is the following one (Theorem \ref{THE:Gamma continue} below). The assumption of $\E$ being stable with respect to some balanced form is no more necessary for Theorem \ref{THE:4}.

\begin{theoremIntro}\label{THE:4}
    The function that maps each balanced metric $\Theta$ with respect to which $\E$ is semi-stable and sufficiently smooth to the class modulo $\G$ of the $\Theta$-HYM $h$-unitary connections $\nabla = \partial + \dbar$ such that $(E,\dbar) \cong \Gr_{[\Theta]}(\E)$, is a continuous function with respect to the $\mC^\infty$ topologies.
\end{theoremIntro}

Theorem \ref{THE:2} is a generalisation of \cite[Proposition 1.1]{Clarke_Tipler} and Theorem \ref{THE:3} is a generalisation of \cite[Theorem 1.2]{Clarke_Tipler}, where in both cases, we remove the hypothesis that the group of automorphisms of the graded object is abelian. Theorem \ref{THE:4} is a natural consequence of Theorem \ref{THE:3} and may be useful to study the variations of the moduli space of vector bundles when the metric of the manifold varies. Proposition \ref{PRO:1} may also be more generally useful to study wall-crossing phenomena involving the slope stability.

The method used in this article relies on semi-stable perturbation techniques. When $\E$ is a $[\Theta_0]$-semi-stable sufficiently smooth vector bundle, it can be viewed as a small deformation of its graded object $\Gr(\E) = \Gr_{[\Theta_0]}(\E)$. Then, by Kuranishi theory \cite{Kuranishi}, all the small deformations of $\Gr(\E)$ belong, up to a gauge transformation, in the Kuranishi slice which is a germ of finite dimensional complex manifold, on which the group $K$ of unitary automorphisms of $\Gr(\E)$ acts naturally. Moreover, for all $\Theta$ $\mC^0$-close to $[\Theta_0]$, there is a Kähler form $\Omega_\Theta^D$ on this germ depending of $\Theta$ such that the action of $K$ is Hamiltonian with respect to this form. Additionally, the zeroes of the associated moment map correspond to connections $\nabla$ such that the contraction of the curvature $\frac{\i F_\nabla \wedge \omega^{n - 1}}{\omega^n}$ with respect to $\Theta = \frac{\omega^{n - 1}}{(n - 1)!}$ is orthogonal to the Lie algebra $\frak{k}$ of $K$. We want this contraction to be a constant homothety in order to obtain a HYM connection.

From here, there are two methods. The first one is to find a zero of this moment map using GIT theory, and then perturb the induced connection to obtain a new connection whose contraction of the curvature is a constant homothety. This method is used by Sektnan--Tipler \cite{Sektnan_Tipler}, Dervan--McCarthy--Sektnan \cite{DMS} and Ortu--Sektnan \cite{Ortu_Sektnan} for example. The second method is to first perturb the Kuranishi slice so the holomorphic structures in the perturbed slice induce connections whose contraction of the curvature belongs to $\frak{k}$. Then, we find a zero of the associated perturbed moment map, whose induced connection is directly HYM. The main issue with this second method is that the perturbed slice is no more a germ of complex manifold, it is only symplectic. Nevertheless, this second method is used by Székelyhidi \cite{Gabor}, Clarke--Tipler \cite{Clarke_Tipler} and Ortu \cite{Ortu} for example. In this paper, we use the second method too with the moment map flow. It is versatile and can probably be used in similar settings whenever there is a PDE with a moment map interpretation and wall-crossing phenomena.

These results naturally raise the question of whether or not they generalise to the case of Gauduchon metrics. A metric $g$ is called Gauduchon if $\partial\dbar\omega^{n - 1} = 0$. All complex manifolds admit Gauduchon metrics making it a very general situation. The algebraic part of this paper used to prove Theorem \ref{THE:2} applies \textit{mutatis mutandis} when we consider only Gauduchon metrics. The analytic part used to prove Theorems \ref{THE:3} and \ref{THE:4} however do not generalise immediately due to the lack of moment map interpretation in the non-balanced case in Section \ref{SEC:Résulats analytiques}. It is expected however that Theorems \ref{THE:3} and \ref{THE:4} are still true with Gauduchon metrics.

\paragraph{Acknowledgement}

I thank my PhD supervisor Carl Tipler for pointing out this problem, giving me the idea of the perturbed Kuranishi slice and for his feedbacks on the paper. I also thank Lars Martin Sektnan for his feedbacks on this paper and the talk we had about it and about his work with Annamaria Ortu on constant scalar curvature metrics. I thank Matei Toma for pointing out a mistake in the first draft of this paper.

\section{Generalities about balanced metrics and their Laplace-Beltrami operators}\label{SEC:Généralités}

\subsection{Operations on compact complex manifolds}

Let $(X,g)$ be a Hermitian compact complex manifold of dimension $n \geq 2$ and let $\omega$ be the associated $(1,1)$-form (not necessarily closed). Then,
$$
\Theta = \frac{\omega^{n - 1}}{(n - 1)!},
$$
is a positive $(n - 1,n - 1)$-form in the following sense. In a local frame $(z_1,\ldots,z_n)$, we can define for each $1 \leq i,j \leq n$,
$$
\widehat{dz_i \wedge d\overline{z_j}} = dz_1 \wedge d\overline{z_1} \wedge \cdots \wedge dz_{i - 1} \wedge d\overline{z_{i - 1}} \wedge d\overline{z_i} \wedge \cdots \wedge d\overline{z_{j - 1}} \wedge dz_j \wedge d\overline{z_{j + 1}} \wedge \cdots \wedge dz_n \wedge d\overline{z_n}.
$$
Then, the local expression of $\Theta$ has the form,
$$
\Theta = \sum_{1 \leq i,j \leq n} \Theta_{ij}\i^{n - 1}\widehat{dz_i \wedge d\overline{z_j}},
$$
Positivity  of $\Theta$ means that the matrix formed by the the $\Theta_{ij}$ is Hermitian positive definite at each point, and this does not depend on the choice of the local frame. By \cite[Equation (4.8)]{Michelsohn}, $\omega$ (thus $g$) is entirely determined by $\Theta$ and depends smoothly on it. In other words, any positive $(n - 1,n - 1)$ form $\Theta$ gives rise to a Hermitian metric $g$ on $X$ which depends smoothly on $\Theta$. Let $\Vol_\Theta = \frac{\omega^n}{n!}$ be the natural associated volume form. For all $\Theta > 0$, we can define the trace operator $\Lambda_\Theta$ on complex forms as $\star^{-1} \circ (\omega \wedge \cdot) \circ \star$ where $\star$ is the Hodge star. In particular, on $(1,1)$ forms $\alpha$, it is characterised by the equality,
$$
(\Lambda_\Theta\alpha)\Vol_\Theta = \alpha \wedge \Theta.
$$
This operator naturally extends to forms with values in a complex vector bundle $E$ on $X$. Let $(E,h)$ be a complex Hermitian vector bundle on $X$. Following for example \cite[Definition 4.1.11]{Huybrechts}, let us recall the natural Hermitian product on $\Omega^{p,q}(X,E)$ with respect to the metric $\Theta$,
\begin{equation}\label{EQ:Produit hermitien}
    \scal{\alpha}{\beta}_\Theta = \int_X (\alpha|\beta)_{\Theta,h}\Vol_\Theta,
\end{equation}
where $(\cdot|\cdot)_{\Theta,h}$ is the natural Hermitian product on $\Lambda^{p,q}T^*X \otimes E$ depending on $\Theta$ and $h$, and we see $\alpha$ and $\beta$ as sections of this bundle. It makes the decomposition $\Omega^*(X,E) = \bigoplus_{0 \leq p,q \leq n} \Omega^{p,q}(X,E)$ orthogonal. In particular, on sections $\xi,\eta$ and on $1$-forms $\alpha,\beta$, we have,
$$
\scal{\xi}{\eta}_\Theta = \int_X h(\xi,\eta)\Vol_\Theta, \qquad \scal{\alpha}{\beta}_\Theta = \int_X h(\alpha,J\beta) \wedge \Theta.
$$
Here, $J$ is the natural complex structure on $\Omega^*(X,E)$ given by $J = \i^{p - q}$ on $(p,q)$-forms. It enables us to consider adjoints of operators with respect to these Hermitian products. Let $\nabla = \partial + \dbar$ be an integrable $h$-unitary connection on $E$ (\textit{i.e.} $\dbar^2 = 0$ and $\nabla$ is the Chern connection associated to $(\dbar,h)$). From now on, all connections are assumed to be $h$-unitary so there is a one-to-one correspondence between Dolbeault operators and connections on $E$ given by the Chern correspondence. Then, we can define the Laplacian operators with respect to $\Theta$ and $\nabla$,
$$
\Delta_{\Theta,\partial} = \partial\partial^* + \partial^*\partial, \qquad \Delta_{\Theta,\dbar} = \dbar\dbar^* + \dbar^*\dbar, \qquad \Delta_{\Theta,\nabla} = \nabla\nabla^* + \nabla^*\nabla,
$$
The second one is called the Dolbeault Laplacian and the third one the Laplace-Dolbeault Laplacian. These three operators are real self-adjoint and elliptic of order $2$. See for example \cite[Definition (3.14)]{Demailly}.

\subsection{When $X$ is balanced}

The Hermitian metric $g$ on $X$ is said to be a \textit{balanced metric} if the associated $(1,1)$-form $\omega$ is such that $\omega^{n - 1}$ is closed. Balanced metrics where introduced by Michelsohn in \cite{Michelsohn}, see \cite{Fu} for a reference. Notice that $\Theta = \frac{\omega^{n - 1}}{(n - 1)!}$ is closed (\textit{i.e.} $\omega$ is balanced) if and only if $\omega$ is co-closed \cite[Proposition 1]{Gauduchon}. With a slight abuse of notation, we call $\Theta$ a balanced metric in this case. When such a metric exists, we say that $X$ is balanced. A characterisation of balanced manifolds is given by \cite[Theorem A]{Michelsohn}. Similarly to the Kähler cone, we can introduce the balanced cone,
$$
\mathcal{B}_X = \{[\Theta] \in H^{n - 1,n - 1}(X,\R)|\Theta \textrm{ positive $(n - 1,n - 1)$-form}\},
$$
where the $H^{p,p}(X,\R) = H^{p,p}(X,\C) \cap H^{2p}(X,\R)$ are the real Dolbeault cohomology groups.
\begin{remark}
In the literature about balanced manifolds, the balanced cone usually contains elements of the Bott-Chern cohomology groups instead of the real Dolbeault ones, where the boundary morphisms are $\i\partial\dbar$ instead of $d$. Here, we only need to work modulo $d$. Since $\partial\dbar = \frac{1}{2}(\partial - \dbar)d$, the real Dolbeault cohomology groups are naturally quotients of the Bott-Chern ones. Therefore, this article's balanced cone can be seen as a projection of the classical balanced cone on a quotient space.
\end{remark}
From now on, we assume that $\Theta$ is balanced.

\begin{lemma}\label{LEM:Identités de Kähler}
    The following Kähler identities hold on the space $\Omega^1(X,E)$,
    \begin{align*}
        \Lambda_\Theta\partial & = \i\dbar^*,\\
        \Lambda_\Theta\dbar & = -\i\partial^*,
    \end{align*}
    and the following hold on the space of smooth sections $\Omega^0(X,E)$,
    \begin{align*}
        \Delta_{\Theta,\partial} & = \i\Lambda_\Theta\dbar\partial,\\
        \Delta_{\Theta,\dbar} & = -\i\Lambda_\Theta\partial\dbar,\\
        \Delta_{\Theta,\partial} + \Delta_{\Theta,\dbar} & = \Delta_{\Theta,\nabla} = \i\Lambda_\Theta(\dbar\partial - \partial\dbar),\\
        \Delta_{\Theta,\partial} - \Delta_{\Theta,\dbar} & = \Lambda_\Theta\i F_\nabla,
    \end{align*}
    where $F_\nabla = \nabla^2 = \dbar\partial + \partial\dbar$ is the curvature form of $\nabla$.
\end{lemma}
\begin{proof}
The proof is the same as the standard proof of Kähler identities, but only on sections and $1$-forms. Indeed, if $\alpha$ is a $(0,1)$ form and $s$ a $0$-form with values in $E$,
\begin{align*}
    \scal{\Lambda_\Theta\dbar\alpha}{s}_\Theta & = \int_X h(\Lambda_\Theta\dbar\alpha,s)\Vol_\Theta\\
    & = \int_X h(\dbar\alpha,s) \wedge \Theta\\
    & = \int_X dh(\alpha,s) \wedge \Theta + \int_X h(\alpha,\partial s) \wedge \Theta\\
    & = \int_X h(\alpha,-J\i\partial s) \wedge \Theta \textrm{ because } d\Theta = 0,\\
    & = \scal{\alpha}{-\i\partial s}_\Theta.
\end{align*}
We obtain the second identity. The first one is similar. We deduce that, on smooth sections,
$$
\Delta_{\Theta,\partial} = \partial\partial^* + \partial^*\partial = \partial^*\partial = \i\Lambda_\Theta\dbar\partial,
$$
and similarly,
$$
\Delta_{\Theta,\dbar} = \dbar\dbar^* + \dbar^*\dbar = \dbar^*\dbar = -\i\Lambda_\Theta\partial\dbar.
$$
The equality $\Delta_{\Theta,\nabla} = \Delta_{\Theta,\partial} + \Delta_{\Theta,\dbar}$ always hold by definition even if we don't assume $d\Theta = 0$. The third equality involving Laplacians is the sum of the two first and the fourth one is their difference.
\end{proof}

\begin{proposition}\label{PRO:Noyau Laplacien}
    On smooth sections, $\Delta_{\Theta,\nabla}$ verifies $\ker(\Delta_{\Theta,\nabla}) = \ker(\nabla)$ is the space of $\nabla$-parallel sections and if moreover $\Lambda_\Theta\i F_\nabla = 0$, then $\ker(\Delta_{\Theta,\nabla}) = \ker(\dbar)$ is the space of $\dbar$-holomorphic sections.
\end{proposition}
\begin{proof}
The first equality is a direct consequence of the fact that $\Delta_{\Theta,\nabla} = \nabla^*\nabla$ on smooth forms. The second one comes from Lemma \ref{LEM:Identités de Kähler}, which implies that $\Delta_{\Theta,\nabla} = \Delta_{\Theta,\partial} + \Delta_{\Theta,\dbar}$ and if $\Lambda_\Theta\i F_\nabla = 0$, $\Delta_{\Theta,\partial} = \Delta_{\Theta,\dbar}$ hence $\Delta_{\Theta,\nabla} = 2\Delta_{\Theta,\dbar} = 2\dbar^*\dbar$.
\end{proof}

By abuse of notation, we still call $\nabla = \partial + \dbar$ the natural connection on $\End(E)$ that arises from $\nabla$ on $E$.

\subsection{Gauge action and its derivative}

Let us to introduce the \textit{complex gauge group} $\G^\C$ of smooth sections of $\End(E)$ which are isomorphisms on each fibre. It is an infinite dimensional complex Lie group and $\Lie(\G^\C) = \Omega^0(X,\End(E))$. This group acts on connections on $E$ with,
$$
f \cdot \partial = f^{-1*} \circ \partial \circ f^*, \qquad f \cdot \dbar = f \circ \dbar \circ f^{-1}, \qquad f \cdot \nabla = f \cdot \partial + f \cdot \dbar.
$$
Notice that this action can also be written by using the Dolbeault operator on $\End(E)$ induced by $\dbar$ (we call it $\dbar$ too).
\begin{equation}\label{EQ:Action de gauge}
    f \cdot \dbar = f \circ \dbar \circ f^{-1} = \dbar + f\dbar(f^{-1}).
\end{equation}
The adjoints are computed with respect to the metric $h$. Notice that this action preserves the integrability condition of the Dolbeault operator $\dbar^2 = 0$ and it preserves the fact to be $h$-unitary. When $\dbar_1 = f \cdot \dbar_2$, the holomorphic vector bundles $\E_1 = (E,\dbar_1)$ and $\E_2 = (E,\dbar_2)$ are isomorphic because $f : \E_2 \rightarrow \E_1$ is a biholomorphism. Conversely, such an isomorphism gives rise to a gauge equivalence between the Dolbeault operators of $\E_1$ and $\E_2$.

Similarly, we call $\G \subset \G^\C$ the \textit{unitary gauge group} defined as the set of all $u \in \G^\C$ such that $uu^* = \Id_E$. It is an infinite dimensional real Lie group with $\Lie(\G) = \i\Omega^0(X,\End_H(E,h))$ where $\End_H(E,h) \subset \End(E)$ is the smooth real bundle of Hermitian endomorphisms with respect to $h$. Notice that the action of $\G$ on connections is given by,
$$
u \cdot \nabla = u \circ \nabla \circ u^*,
$$
thus the associated curvature verifies,
\begin{equation}\label{EQ:Action gauge unitaire}
    F_{u \cdot \nabla} = uF_\nabla u^*.
\end{equation}

We will be interested in finding a Dolbeault operator $\dbar$ on $E$ which satisfies the property of being Hermitian Yang--Mills such that $(E,\dbar)$ is isomorphic to a given holomorphic vector bundle $\E = (E,\dbar_E)$. Therefore, we shall search $\dbar$ in the $\G^\C$-orbit of $\dbar_E$. For this, it will be useful to compute infinitesimal variations of $f \mapsto f \cdot \dbar_E$.

\begin{proposition}\label{PRO:Calcul dPsi(0)}
    For any $h$-unitary connection $\nabla$, the map,
    $$
    \Psi : \fonction{\Omega^0(X,\End_H(E,h))}{\Omega^0(X,\End_H(E,h))}{s}{\Lambda_\Theta\i F_{\e^s \cdot \nabla}},
    $$
    verifies $d\Psi(0) = \Delta_{\Theta,\nabla}$ (on $\End(E)$). Notice that since $\Delta_{\Theta,\nabla}$ is a real operator, it preserves $\End_H(E,h)$.
\end{proposition}
\begin{proof}
We have for all smooth Hermitian sections $s$ of $\End(E)$,
$$
\e^s \cdot \nabla = \nabla + \e^s\dbar(\e^{-s}) + \e^{-s}\partial(\e^s),
$$
hence,
$$
F_{\e^s \cdot \nabla} = F_\nabla + \nabla\alpha_s + \alpha_s \wedge \alpha_s,
$$
where $\alpha_s = \e^s\dbar(\e^{-s}) + \e^{-s}\partial(\e^s) \in \i\Omega^1(X,\End_H(E,h))$. We have $\alpha_0 = 0$ and if $v \in T_s\Omega^0(X,\End_H(E,h)) = \Omega^0(X,\End_H(E,h))$,
$$
\frac{\partial}{\partial s}|_{s = 0}(\e^s\dbar(\e^{-s})) \cdot v = \frac{\partial}{\partial s}|_{s = 0}(\e^s) \cdot v\dbar(\e^{-0}) + \e^0\dbar\left(\frac{\partial}{\partial s}|_{s = 0}(\e^{-s}) \cdot v\right) = -\dbar v.
$$
Hence $\frac{\partial}{\partial s}|_{s = 0}\alpha_s \cdot v = \partial v - \dbar v$ symmetrically. Since $\alpha_s \wedge \alpha_s$ is quadratic and $\alpha_0 = 0$, its derivative at $0$ vanishes thus,
$$
\frac{\partial}{\partial s}|_{s = 0}\i F_{\e^s \cdot \nabla} = \i\nabla(\partial v - \dbar v) = \i(\dbar\partial - \partial\dbar)v.
$$
After applying $\Lambda_\Theta$, by Lemma \ref{LEM:Identités de Kähler}, we deduce that $d\Psi(0) = \Delta_{\Theta,\nabla}$.
\end{proof}

\section{Finiteness results about slope stability and the Hermitian Yang--Mills equation}\label{SEC:Finitude des c_1}

\subsection{Slope stability}

Any holomorphic vector bundle can be seen as the locally free coherent sheaf of its holomorphic sections. In this subsection, we study more generally coherent sheaves on $X$ and their slope stablity with respect to a positive real Dolbeault class $[\Theta]$.
\begin{definition}
    We define the $[\Theta]$-slope $\mu_{[\Theta]}$ of a non-zero torsion-free coherent sheaf $\E$ as
    $$
    \mu_{[\Theta]}(\E) = \frac{c_1(\E) \cup [\Theta]}{\rk(\E)}.
    $$
    Following Mumford \cite{Mumford}, we say that,
    \begin{itemize}
        \item $\E$ is $[\Theta]$-stable if for all coherent sub-sheaves $\mathcal{S} \subset \E$ with $0 < \rk(\mathcal{S}) < \rk(\E)$,
        $$
        \mu_{[\Theta]}(\mathcal{S}) < \mu_{[\Theta]}(\E).
        $$
        \item $\E$ is $[\Theta]$-semi-stable if for all coherent sub-sheaves $\mathcal{S} \subset \E$ with $0 < \rk(\mathcal{S}) < \rk(\E)$,
        $$
        \mu_{[\Theta]}(\mathcal{S}) \leq \mu_{[\Theta]}(\E).
        $$
        \item $\E$ is $[\Theta]$-polystable if $\E$ is a direct sum of stable coherent sheaves of the same $[\Theta]$-slope.
    \end{itemize}
\end{definition}

Clearly, stability implies polystability. When $\E$ can be written as a direct sum $\bigoplus_{i = 1}^m \E_i$ (or more generally as an extension of the $\E_i$), then $c_1(\E) = \sum_{i = 1}^m c_1(\E_i)$ thus $\mu_{[\Theta]}(\E)$ is a weighted average of the $\mu_{[\Theta]}(\E_i)$ where the weights are the $\rk(\E_i)$. Therefore, polystability implies semi-stability.

In the case where $\E$ is only $[\Theta]$-semi-stable, we may use a Jordan-Hölder filtration to reduce to the polystable case.
\begin{proposition}\label{PRO:FJH}
    If $\E$ is a $[\Theta]$-semi-stable vector bundle, then it admits a Jordan-Hölder filtration, \textit{i.e.} a filtration by coherent sub-sheaves $0 \subsetneq \F_1 \subsetneq \cdots \subsetneq \F_m = \E$ such that for all $k$, $\mu_{[\Theta]}(\F_k) = \mu_{[\Theta]}(\E)$ and each $\G_k = \F_k/\F_{k - 1}$ is torsion-free and $[\Theta]$-stable. Moreover, the graded object,
    $$
    \Gr_{[\Theta]}(\E) = \bigoplus_{k = 1}^m \G_k,
    $$
    is unique up to isomorphism outside of a codimension $2$ sub-analytic space of $X$. In particular, its reflexive closure $\Gr_{[\Theta]}(\E)^{\lor\lor}$ is unique up to isomorphism.
\end{proposition}
\begin{proof}
This is a standard result. See for example \cite[Theorem 1.6.7, Corollary 1.6.10]{Huybrechts_Lehn}.
\end{proof}

\begin{definition}
    When $\E$ is locally free \textit{i.e.} a vector bundle, it is said to be \textit{sufficiently smooth} if its graded object $\Gr_{[\Theta]}(\E)$ can be chosen to be locally free too, or equivalently, each $\G_i$ can be chosen to be locally free.
\end{definition}
\begin{remark}
    When $F \subset E$ are smooth vector bundles, the Hermitian metric $h$ on $E$ gives rise to a diffeomorphism between $E/F$ and $F^\bot$. In other words, any exact sequence of smooth vector bundles splits. In particular, $\E$ and $\Gr_{[\Theta]}(\E)$ have the same smooth structure when $\E$ is a sufficiently smooth vector bundle.
\end{remark}
\begin{remark}
    $\E$ is $[\Theta]$-polystable if and only if it is $[\Theta]$-semi-stable and $\Gr_{[\Theta]}(\E) \cong \E$.
\end{remark}

\subsection{Stable and semi-stable polyhedral cones}

Let $\E$ be a torsion-free coherent sheaf. We are interested here in the structure of the set of $\Theta$ such that $\E$ is $[\Theta]$-stable (resp. $[\Theta]$-semi-stable). Let us introduce,
$$
\mC_s(\E) = \{\Theta > 0|d\Theta = 0 \textrm{ and $\E$ is $[\Theta]$-stable}\}, \qquad \mC_{ss}(\E) = \{\Theta > 0|d\Theta = 0 \textrm{ and $\E$ is $[\Theta]$-semi-stable}\},
$$
and their real Dolbeault projection counterparts,
$$
\mC_s^{\mathrm{D}}(\E) = \{[\Theta]|\Theta \in \mC_s(\E)\} = \{[\Theta] \in \mathcal{B}_X|\E \textrm{ is $[\Theta]$-stable}\},
$$
$$
\mC_{ss}^{\mathrm{D}}(\E) = \{[\Theta]|\Theta \in \mC_{ss}(\E)\} = \{[\Theta] \in \mathcal{B}_X|\E \textrm{ is $[\Theta]$-semi-stable}\}.
$$
When $\mathcal{S} \subset \E$ verifies $0 < \rk(\mathcal{S}) < \rk(\E)$, let,
$$
l_{\mathcal{S}} : \fonction{H^{n - 1,n - 1}(X,\R)}{\R}{[\Theta]}{\mu_{[\Theta]}(\E/\mathcal{S}) - \mu_{[\Theta]}(\mathcal{S}) = \left(\frac{c_1(\E/\mathcal{S})}{\rk(\E/\mathcal{S})} - \frac{c_1(\mathcal{S})}{\rk(\mathcal{S})}\right) \wedge [\Theta]}.
$$
Each $l_{\mathcal{S}}$ is a linear form on $H^{n - 1,n - 1}(X,\R)$ and the see-saw property of slopes (see for example \cite[Proposition 5.7.4]{Kobayashi}) implies that for each $[\Theta] \in \mathcal{B}_X$, $\E$ is $[\Theta]$-stable (resp $[\Theta]$-semi-stable) if and only if for all $\mathcal{S}$, $l_{\mathcal{S}}([\Theta]) > 0$ (resp. $\geq 0$). In other words, we have,
$$
\mC_s^{\mathrm{D}}(\E) = \mathcal{B}_X \cap \bigcap_{\mathcal{S} \subset \E}\{l_{\mathcal{S}} > 0\}, \qquad \mC_{ss}^{\mathrm{D}}(\E) = \mathcal{B}_X \cap \bigcap_{\mathcal{S} \subset \E}\{l_{\mathcal{S}} \geq 0\},
$$
In particular, these four sets all are convex cones.

\subsection{Local finiteness}

The purpose of this subsection is to show that the number of linear inequalities $l_{\mathcal{S}} > 0$ to verify in order to obtain stability is locally finite, so $\mC_{ss}^{\mathrm{D}}(\E)$ and $\mC_s^{\mathrm{D}}(\E)$ are locally polyhedral convex cones.

\begin{lemma}\label{LEM:Finitude locale}
    Let $V$ be a finite dimensional real vector space and $D \subset V$ a non-empty discrete subset. Let $U \subset V^\lor$ be an open subset of the dual of $V$ and such that for all $\varphi \in U$, $\varphi(D) \subset \R$ is bounded from above. Then, for all compact $K \subset U$, there is a finite set $F \subset D$ that may depend on $K$ such that for all $\varphi \in K$, $\sup(\varphi(D)) = \max(\varphi(F))$.
\end{lemma}
\begin{proof}\ \\
\vspace{1mm}
\noindent\textbf{Step 1 :} For all real number $a$, for all $K \subset U$ compact, the set $F_{K,a} = \bigcup_{\varphi \in K} \{\varphi \geq a\} \cap D$ is finite.
\vspace{1mm}

Indeed, $F_{K,a} \subset D$ is discrete so if it is infinite, then it is unbounded. In this case, there is a sequence $(v_m)$ of elements of $F_{K,a}$ such that $\norme{v_m} \rightarrow +\infty$ where $\norme{\cdot} = \sqrt{\scal{\cdot}{\cdot}}$ is any Euclidean norm on $V$. Up to extracting, $\frac{v_m}{\norme{v_m}} \rightarrow v_\infty \in V$ is in the unit sphere. For all fixed $m$, $v_m \in \bigcup_{\varphi \in K} \{\varphi \geq a\}$ so there is a $\varphi_m \in K$ such that $\varphi_m(v_m) \geq a$. Up to extracting again, $\varphi_m \rightarrow \varphi_\infty \in K$.

Let $\epsilon > 0$ and $\varphi = \varphi_\infty + \epsilon\scal{v_\infty}{\cdot}$. Since $U$ is open, for $\epsilon$ small enough, $\varphi \in U$. For all integer $m$,
\begin{align*}
    \varphi(v_m) & = \varphi_\infty(v_m) + \epsilon\scal{v_\infty}{v_m}\\
    & = \varphi_m(v_m) + \mathrm{O}(\norme{\varphi_m - \varphi_\infty}\norme{v_m}) + \epsilon\scal{\frac{v_m}{\norme{v_m}} + \mathrm{o}(1)}{v_m}\\
    & \geq a + \epsilon\norme{v_m} + \mathrm{o}(\norme{v_m})\\
    & \tend{m}{+\infty} +\infty.
\end{align*}
This contradicts the boundedness from above of $\varphi(D)$. It proves that $F_{K,a}$ is indeed finite. In particular, when $K = \{\varphi\}$ is a singleton and $a = \sup(\varphi(D)) - 1$, we obtain that for all $\varphi \in U$, the set,
$$
\{\varphi \geq \sup(\varphi(D)) - 1\} \cap D,
$$
is finite. In particular, it proves that the $\sup$ is reached (and it is reached finitely many times).

\vspace{1mm}
\noindent\textbf{Step 2 :} The map $\varphi \mapsto \max(\varphi(D))$ is continuous on $U$.
\vspace{1mm}

Let $\varphi_0 \in U$ and $K \subset U$ be a compact set such that $\varphi_0 \in \inter{K}$. Let $v \in D$ such that $\max(\varphi_0(D)) = \varphi_0(v)$ and $a = \max(\varphi_0(D)) - 1$. Clearly, $v \in F_{K,a}$ and when $\varphi$ is close enough to $\varphi_0$, then $\varphi \in K$ and $\varphi(v) \geq a$. Thus,
$$
\max(\varphi(D)) \geq \varphi(v) \geq a,
$$
so the $w \in D$ that verify $\max(\varphi(D)) = \varphi(w)$ belong to $F_{K,a}$. In other words, for all $\varphi$ close enough to $\varphi_0$, $\max(\varphi(D)) = \max(\varphi(F_{K,a}))$.

Since $F_{K,a}$ is finite, $\varphi \mapsto \max(\varphi(F_{K,a}))$ is continuous at $\varphi_0$, proving the wanted result.

\vspace{1mm}
\noindent\textbf{Step 3 :} Conclusion.
\vspace{1mm}

Let $K \subset U$ be an arbitrary compact set. By the continuity result of step 2,
$$
a = \min_{\varphi \in K}(\max(\varphi(D))),
$$
is well defined and by the result of step 1,
$$
F = F_{K,a} = \bigcup_{\varphi \in K} \{\varphi \geq a\} \cap D,
$$
is finite. For all $\varphi \in K$, if $v \in D$ is such that $\max(\varphi(D)) = \varphi(v)$, then $\varphi(v) \geq a$ so $v \in F$. We deduce that $\sup(\varphi(D))$ is reached by an element in the finite set $F$, which is the wanted result.
\end{proof}

We can now prove Introduction's Proposition \ref{PRO:1}.

\begin{corollary}\label{COR:Finitude locale}
    For all torsion-free coherent sheaf $\E$ and all compact $K \subset \mathcal{B}_X$, there is a finite family $\mathcal{S}_1,\ldots,\mathcal{S}_p$ of coherent sub-sheaves of $\E$ with for all $k$, $0 < \rk(\mathcal{S}_k) < \rk(\E)$ such that for all $[\Theta] \in K$, $\E$ is $[\Theta]$-stable (resp. $[\Theta]$-semi-stable) if and only if for all $k$, $l_{\mathcal{S}_k}([\Theta]) > 0$ (resp. $l_{\mathcal{S}_k}([\Theta]) \geq 0$).
\end{corollary}
\begin{proof}
Let $V = H^{1,1}(X,\R)$. By Serre duality, $V^\lor$ is naturally identified with $H^{n - 1,n - 1}(X,\R)$ thus we can see $\mathcal{B}_X$ as an open subset of $V^\lor$. Let,
$$
D = \left\{\frac{c_1(\mathcal{S})}{\rk(\mathcal{S})}\middle|\mathcal{S} \subset \E, 0 < \rk(\mathcal{S}) < \rk(\E)\right\}.
$$
The first Chern classes belong to the lattice $H^2(X,\Z)$ hence,
$$
D \subset \frac{1}{\rk(\E)!}H^2(X,\Z),
$$
is discrete. For each fixed $[\Theta] \in \mathcal{B}_X$, $\{c \cup [\Theta]|c \in D\} \subset \R$ is bounded from above. This is a standard fact uses to build Harder-Narasimhan filtrations (see for example \cite[Lemma 5.7.16]{Kobayashi}).

Therefore, we can use Lemma \ref{LEM:Finitude locale} which tells us that for all compact $K \subset \mathcal{B}_X$, there is a finite set $F \subset D$ such that for all $[\Theta] \in K$,
\begin{equation}\label{EQ:Sup atteint}
    \sup\{c \cup [\Theta]|c \in D\} = \max\{c \cup [\Theta]|c \in F\}.
\end{equation}
Let $F = \{c^1,\ldots,c^p\}$ and for all $k$, $\mathcal{S}_k \subset \E$ with $0 < \rk(\mathcal{S}_k) < \rk(\E)$ such that $\frac{c_1(\mathcal{S}_k)}{\rk(\mathcal{S}_k)} = c^k$. This tuple fits. Indeed, when $[\Theta] \in K$,
\begin{align*}
    \E \textrm{ is $[\Theta]$-stable } & \Leftrightarrow \forall \mathcal{S} \subset \E, 0 < \rk(\mathcal{S}) < \rk(\E) \Rightarrow \frac{c_1(\mathcal{S})}{\rk(\mathcal{S})} \cup [\Theta] < \frac{c_1(\E)}{\rk(\mathcal{E})} \cup [\Theta]\\
    & \Leftrightarrow \sup\{c \cup [\Theta]|c \in D\} < \frac{c_1(\E)}{\rk(\E)} \cup [\Theta] \textrm{ because the $\sup$ is reached,}\\
    & \Leftrightarrow \max\{c \cup [\Theta]|c \in F\} < \frac{c_1(\E)}{\rk(\E)} \cup [\Theta] \textrm{ by (\ref{EQ:Sup atteint}),}\\
    & \Leftrightarrow \forall 1 \leq k \leq p, l_{\mathcal{S}_k}([\Theta]) > 0.
\end{align*}
The same equivalence holds with weak inequalities.
\end{proof}

This last corollary enables us to understand better the geometric properties of the stable and semi-stable cones. For the topological structures, we endow the space of $(n - 1,n - 1)$ closed forms with the $\mC^0$ topology. As a consequence of De Rham's theorem, the $\mC^0$ topology makes the space of exact forms closed in the space of closed forms thus its natural projection on $H^{n - 1,n - 1}(X,\R)$ is continuous.

Clearly, $\mC_s(\E)$ is a $\mC^0$-open convex cone in the space of closed $(n - 1,n - 1)$ forms. $\mC_s^{\mathrm{D}}(\E)$ is an open convex cone in $H^{n - 1,n - 1}(X,\R)$. $\mC_{ss}(\E)$ is a $\mC^0$-closed convex cone in the space of balanced metrics. $\mC_{ss}^{\mathrm{D}}(\E)$ is a convex cone in $H^{n - 1,n - 1}(X,\R)$ and is closed in $\mathcal{B}_X$. All of these are trivial consequences of the local finiteness result of Corollary \ref{COR:Finitude locale} and of the continuity of the natural projection on $H^{n - 1,n - 1}(X,\R)$. Moreover, we have the more precise version of Introduction's Theorem \ref{THE:2}.
\begin{theorem}\label{THE:Structures cônes}
    On a balanced manifold $X$, for all torsion-free coherent sheaf $\E$, the stable and semi-stable cones satisfy the following properties,
    \begin{itemize}
        \item $\mC_s^{\mathrm{D}}(\E)$ and $\mC_{ss}^{\mathrm{D}}(\E)$ are locally polyhedral convex cones.
        \item If there is a coherent $\mathcal{S} \subset \E$ such that $0 < \rk(\mathcal{S}) < \rk(\E)$ and $\frac{c_1(\mathcal{S})}{\rk(\mathcal{S})} = \frac{c_1(\E)}{\rk(\E)}$, then $\mC_s(\E) = \mC_s^{\mathrm{D}}(\E) = \emptyset$.
        \item If there is no coherent $\mathcal{S} \subset \E$ such that $0 < \rk(\mathcal{S}) < \rk(\E)$ and $\frac{c_1(\mathcal{S})}{\rk(\mathcal{S})} = \frac{c_1(\E)}{\rk(\E)}$, then $\inter{\mC_{ss}(\E)} = \mC_s(\E)$ and $\inter{\mC_{ss}^{\mathrm{D}}(\E)} = \mC_s^{\mathrm{D}}(\E)$.
        \item If $\mC_s^{\mathrm{D}}(\E) \neq \emptyset$, $\overline{\mC_s(\E)} \cap \{\Theta > 0\} = \mC_{ss}(\E)$ and $\overline{\mC_s^{\mathrm{D}}(\E)} \cap \mathcal{B}_X = \mC_{ss}^{\mathrm{D}}(\E)$.
        \item For all $[\Theta_0] \in \mC_{ss}^{\mathrm{D}}(\E)$, if $[\Theta] \in \mathcal{B}_X$ is close enough to $[\Theta_0]$, $[\Theta] \in \mC_s^{\mathrm{D}}(\E)$ (resp. $[\Theta] \in \mC_{ss}^{\mathrm{D}}(\E)$) if and only if for all $\mathcal{S} \subset \E$ such that $l_{\mathcal{S}}([\Theta_0]) = 0$, $l_{\mathcal{S}}([\Theta]) > 0$ (resp. $l_{\mathcal{S}}([\Theta]) \geq 0$).
    \end{itemize}
\end{theorem}
\begin{proof}
By Corollary \ref{COR:Finitude locale}, if $K \subset \mathcal{B}_X$ is compact, there is an integer $p$ and a $p$-tuple $\mathcal{S}_1,\ldots,\mathcal{S}_p \subsetneq \E$ of coherent sheaves such that for all $k$, $0 < \rk(\mathcal{S}_k) < \rk(\E)$ and,
$$
\mC_s^{\mathrm{D}}(\E) \cap K = \bigcap_{k = 1}^p \{l_{\mathcal{S}_k} > 0\} \cap K, \qquad \mC_{ss}^{\mathrm{D}}(\E) \cap K = \bigcap_{k = 1}^p \{l_{\mathcal{S}_k} \geq 0\} \cap K.
$$
It gives then local polyhedral structures.

Moreover, since the pairing $H^{n - 1,n - 1}(X,\R) \times H^{1,1}(X,\R) \rightarrow \R$ is non-degenerate by Serre duality, for each $\mathcal{S}$,
$$
l_{\mathcal{S}} = 0 \Leftrightarrow \frac{c_1(\mathcal{S})}{\rk(\mathcal{S})} = \frac{c_1(\E)}{\rk(\E)}.
$$
From here, the results about the finite dimensional cones follow from basic topology and convex geometry, and the results about the infinite dimensional cones follow from the continuity of the projection onto $H^{n - 1,n - 1}(X,\R)$.
\end{proof}

\subsection{The sufficiently smooth semi-stable cone}

We end this section by studying the set of all $[\Theta] \in \mC_{ss}(\E)$ such that $\E$ is sufficiently smooth. In particular, we obtain the openness of this property.

\begin{lemma}\label{LEM:Gr(Gr)}
    Let $\E$ be a torsion-free coherent sheaf and $K \subset \mathcal{B}_X$ compact and convex. By Theorem \ref{THE:Structures cônes}, $C = \mC_{ss}(\E) \cap K$ is a convex polyhedral cone. Thus for all $[\Theta] \in C$, there is a unique face $F_{[\Theta]} \subset C$ such that $[\Theta]$ belongs to the relative interior of $F_{[\Theta]}$. If $[\Theta_1]$ and $[\Theta_2]$ are in $C$ and $F_{[\Theta_1]} \subset F_{[\Theta_2]}$, then $\Gr_{[\Theta_2]}(\E)$ is $[\Theta_1]$-semi-stable and,
    $$
    \Gr_{[\Theta_1]}(\Gr_{[\Theta_2]}(\E)) \cong \Gr_{[\Theta_1]}(\E).
    $$
    In particular, if $F_{[\Theta_1]} = F_{[\Theta_2]}$, $\Gr_{[\Theta_1]}(\E) \cong \Gr_{[\Theta_2]}(\E)$.
\end{lemma}
\begin{proof}
Let $\mathcal{S}_1,\ldots,\mathcal{S}_p \subset \E$ be a finite family of coherent sheaves given by Corollary \ref{COR:Finitude locale} such that $C = \bigcap_{k = 1}^p \{l_{\mathcal{S}_k} \geq 0\} \cap K$. We may assume that the $l_{\mathcal{S}_k}$ are pairwise distinct. Up to rearranging the $\mathcal{S}_k$, we may also assume that the faces $F_{[\Theta_2]}$ and $F_{[\Theta_1]} \subset F_{[\Theta_2]}$ are defined by the equations,
$$
F_{[\Theta_1]} = \bigcap_{k = 1}^j \{l_{\mathcal{S}_k} = 0\} \cap \bigcap_{k = j + 1}^p \{l_{\mathcal{S}_k} \geq 0\} \cap K, \qquad F_{[\Theta_2]} = \bigcap_{k = 1}^i \{l_{\mathcal{S}_k} = 0\} \cap \bigcap_{k = i + 1}^p \{l_{\mathcal{S}_k} \geq 0\} \cap K,
$$
for some $0 \leq i \leq j \leq p$. Moreover, since $[\Theta_1]$ (resp. $[\Theta_2]$) belongs to the relative interior of $F_{[\Theta_1]}$ (resp. $F_{[\Theta_2]}$), we have $l_{\mathcal{S}_k}([\Theta_1]) > 0$ when $k > j$ (resp. $l_{\mathcal{S}_k}([\Theta_2]) > 0$ when $k > i$).

Let,
$$
0 \subsetneq \F_1 \subsetneq \cdots \subsetneq \F_m = \E,
$$
be a Jordan-Hölder filtration for $\E$ with respect to $[\Theta_2]$. We have,
\begin{equation}\label{EQ:Expression gradué}
    \Gr_{[\Theta_2]}(\E) = \bigoplus_{k = 1}^m \G_k,
\end{equation}
where $\G_k = \F_k/\F_{k - 1}$. By definition, for all $k$, $l_{\F_k}([\Theta_2]) = 0$. Since the $\F_k$ destabilise $\E$ with respect to $[\Theta_2] \in K$, there must exist an integer $1 \leq q \leq p$ such that $\frac{c_1(\F_k)}{\rk(\F_k)} = \frac{c_1(\mathcal{S}_q)}{\rk(\mathcal{S}_q)}$. The equality $l_{\F_k}([\Theta_2]) = l_{\mathcal{S}_q}([\Theta_2]) = 0$ implies that $q \leq i$. In particular, $q \leq j$ so for all $k$, $l_{\F_k}([\Theta_1]) = 0$. It implies that each $\F_k$ is $[\Theta_1]$-semi-stable by $[\Theta_1]$-semi-stability of $\E$. Therefore, for all $k$, $l_{\G_k}([\Theta_1]) = 0$ and $\G_k$ is $[\Theta_1]$-semi-stable as a quotient of $[\Theta_1]$-semi-stable sheaves of the same $[\Theta_1]$-slope.

It means that the $\F_k$ form a filtration of $\E$ by semi-stable sub-sheaves with torsion-free semi-stable quotients with respect to $[\Theta_1]$. Let us consider a maximal refinement of this filtration with for all $1 \leq k \leq m$,
$$
\F_{k - 1} = \F_{k - 1,0} \subsetneq \F_{k - 1,1} \subsetneq \cdots \subsetneq \F_{k - 1,p_k} = \F_k = \F_{k,0}.
$$
Then, the $\F_{k,q}$ form a Jordan-Hölder filtration of $\E$ with respect to $[\Theta_1]$ thus on the one hand,
$$
\Gr_{[\Theta_1]}(\E) = \bigoplus_{k = 1}^m\bigoplus_{q = 1}^{p_k} \F_{k - 1,q}/\F_{k - 1,q - 1},
$$
and on the other hand, for each $k$, the $\F_{k - 1,q}/\F_{k - 1}$ form a Jordan-Hölder filtration of $\G_k = \F_k/\F_{k - 1}$ with respect to $[\Theta_1]$ so, by (\ref{EQ:Expression gradué}),
$$
\Gr_{[\Theta_1]}(\Gr_{[\Theta_2]}(\E)) = \bigoplus_{k = 1}^m \Gr_{[\Theta_1]}(\G_k) = \bigoplus_{k = 1}^m\bigoplus_{q = 1}^{p_k} \F_{k - 1,q}/\F_{k - 1,q - 1}.
$$
It proves the lemma.
\end{proof}
Let $\underline{\mC}_{ss}(\E) \subset \mC_{ss}(\E)$ be the set of all $\Theta$ with respect to which $\E$ is semi-stable and sufficiently smooth. Let $\underline{\mC}_{ss}^{\mathrm{D}}(\E) \subset \mC_{ss}^{\mathrm{D}}(\E)$ be its projection on the real Dolbeault cohomology group. Lemma \ref{LEM:Gr(Gr)} has the following consequence,
\begin{corollary}
    $\underline{\mC}_{ss}^{\mathrm{D}}(\E)$ is an open subset of $\mC_{ss}^{\mathrm{D}}(\E)$ whose complement is locally formed of a union of closed faces of $\mC_{ss}^{\mathrm{D}}(\E)$. Similarly, $\underline{\mC}_{ss}(\E)$ is an open subset of $\mC_{ss}(\E)$.
\end{corollary}
\begin{proof}
This is a direct consequence of Lemma \ref{LEM:Gr(Gr)} and the fact that the graded object of a sheaf which is not locally free is itself not locally free (in particular, if $\E$ is not locally free, these cones are empty).
\end{proof}

\section{Local analytic results}\label{SEC:Résulats analytiques}

Let $\Theta_0$ be a balanced metric on $X$ and $\E = (E,\dbar_E)$ a $[\Theta_0]$-semi-stable sufficiently smooth vector bundle. Assume that $\mC_s(\E) \neq \emptyset$ so $\Theta_0 \in \overline{\mC_s(\E)}$ by Theorem \ref{THE:Structures cônes}, and $\E$ is simple. We associate to $\E$ a Jordan-Hölder filtration $0 \subsetneq \F_1 \subsetneq \cdots \subsetneq \F_m = \E$ given by Proposition \ref{PRO:FJH} and we set
$$
\Gr(\E) = \Gr_{[\Theta_0]}(\E) = \bigoplus_{k = 1}^m \G_k,
$$
to be its graded object where the $\G_k = \F_k/\F_{k - 1}$ are its stable components. By Proposition \ref{PRO:Kobayashi--Hitchin}, we may apply a gauge transform to $\dbar_0$ so $(\Gr(\E),h) = (E,\dbar_0,h)$ is $\Theta_0$-HYM. If we apply the same gauge transform to $\dbar_E$, $\dbar_0$ still arises from a Jordan--Hölder filtration of $\E$. Let $\nabla_0 = \partial_0 + \dbar_0$ the associated Chern connection and $F_{\dbar_0} = \nabla_0 \circ \nabla_0$ its curvature form. It satisfies,
\begin{equation}\label{EQ:dbar_0 HYM}
    \Lambda_\Theta\i F_{\dbar_0} = c_{\Theta_0}\Id_E,
\end{equation}
where,
$$
c_{\Theta_0} = \frac{2\pi c_1(E) \cup [\Theta_0]}{\rk(E)\Vol_{\Theta_0}(X)}.
$$
Let for all $k$, $F_k$ be the smooth structure of $\F_k$ and $G_k = F_k/F_{k - 1}$ be the smooth structure of $\G_k$. $h$ induces a natural isometry $G_k = F_k \cap F_{k - 1}^\bot$ so,
$$
E = \bigoplus_{k = 1}^m G_k,
$$
and this decomposition is orthogonal. In particular, the induced decomposition
$$
\Omega^{0,1}(X,\End(E)) = \bigoplus_{1 \leq i,j \leq m} \Omega^{0,1}(X,\Hom(G_j,G_i)),
$$
is orthogonal too for the $L^2$ scalar product. In this decomposition, we can write,
$$
\dbar_E - \dbar_0 = \sum_{1 \leq i,j \leq m} \gamma_{ij}, \qquad \gamma_{ij} \in \Omega^{0,1}(X,\Hom(G_j,G_i)).
$$
For all $1 \leq k \leq m$, the smooth embeddings $F_k \subset E$ are holomorphic when $E$ is endowed with $\dbar_0$ or with $\dbar_E$. It means that for all smooth section $s$ of $F_k$, $\dbar_0s$ and $\dbar_Es$ belong to,
$$
\Omega^{0,1}(\End(F_k)) = \bigoplus_{1 \leq i,j \leq k} \Omega^{0,1}(\Hom(G_j,G_i)).
$$
In particular,
$$
\dbar_Es - \dbar_0s = \sum_{1 \leq i,j \leq m} \gamma_{ij}s = \sum_{i = 1}^m\sum_{j = 1}^k \gamma_{ij}s.
$$
Since this form is always in $\bigoplus_{1 \leq i,j \leq k} \Omega^{0,1}(\Hom(G_j,G_i))$, we deduce that $\gamma_{ij} = 0$ when $1 \leq j \leq k$ and $k + 1 \leq i \leq m$. It is true for all $k$ so $(\gamma_{ij})$ is upper triangular. Moreover, by construction of the quotients $\G_k$,
$$
(G_k,\dbar_{0|G_k}) = \G_k = (G_k,\dbar_{E|G_k}) = (G_k,\dbar_{0|G_k} + \gamma_{k,k}).
$$
We deduce that $\gamma_{kk} = 0$ hence, $\gamma$ is strictly upper triangular \cite[Sub-section 3.2]{Clarke_Tipler} \cite[Sub-section 4.3.1]{DMS}.

Being a balanced metric is an open condition in the space of closed $(n - 1,n - 1)$ forms endowed with the $\mC^0$ topology. During all this article, $U$ will be an $\mC^0$ open neighbourhood of $0$ such that for all $\varepsilon \in U$, $\Theta_\varepsilon = \Theta_0 + \varepsilon$ is positive. With the $\mC^0$ topology, all the natural operations involving $\Theta$ like $\Theta \mapsto \Vol_\Theta$ or $\Theta \mapsto \Lambda_\Theta\alpha$ for some $(1,1)$-form $\alpha$, are smooth. In particular, up to shrinking $U$, the $L^2$ norms $\norme{\cdot}_{\Theta_\varepsilon} = \norme{\cdot}_\varepsilon$ are uniformly equivalent. From now on, when we consider objects that depend on the balanced metric, we replace the subscript $\Theta = \Theta_\varepsilon$ by $\varepsilon$ for simplicity.

\subsection{Perturbed Kuranishi slice and moment map}

We call,
$$
G = \Aut_0(\Gr(\E)) \subset \G^\C,
$$
the group of automorphisms of $\Gr(\E)$ of determinant $1$ and,
$$
K = \Aut_0(\Gr(\E),h) = \G \cap G,
$$
the group of unitary automorphisms of $\Gr(\E)$ of determinant $1$. They are both finite dimensional, $K$ is compact and $G = K^\C$, thus $G$ is a reductive Lie group with maximal compact subgroup $K$. Let,
$$
\frak{g} = \Lie(G) \subset \Lie(\G^\C), \qquad \frak{k} = \Lie(K) = \Lie(\G) \cap \frak{g},
$$
their Lie algebras. Notice that the condition of having a determinant that equals $1$ on $G$ and $K$ implies that all elements of $\frak{g}$ and $\frak{k}$ have a trace that equals $0$. Each of these Lie algebras is identified with its dual space thanks to the scalar products $\scal{\cdot}{\cdot}_\varepsilon$ on,
$$
\Lie(\G^\C) = \Omega^0(X,\End(E)).
$$
We must be careful that this identification depends on $\varepsilon$. Notice also that the automorphisms group and the unitary automorphisms group of $\Gr(\E)$ lie in exact sequences,
$$
\begin{tikzcd}
    1 \ar{r} & G \ar{r} & \Aut(\Gr(\E)) \ar["\det"]{r} & \C^* \ar{r} & 1
\end{tikzcd}
$$
$$
\begin{tikzcd}
    1 \ar{r} & K \ar{r} & \Aut(\Gr(\E),h) \ar["\det"]{r} & \S^1 \ar{r} & 1
\end{tikzcd}
$$
Therefore, their Lie algebras are respectively $\frak{g} \oplus \C\Id_E$ and $\frak{k} \oplus \i\R\Id_E$. The reason why we restrict ourselves to the automorphisms that have a determinant equal to $1$ is that $\C^*\Id_E \subset \Aut(\Gr(\E))$ acts trivially on connections.

Define,
$$
\Omega_\varepsilon^D : (\alpha,\beta) \mapsto \scal{J\alpha}{\beta}_\varepsilon,
$$
to be the symplectic form associated to the Hermitian product introduced at (\ref{EQ:Produit hermitien}) with respect to $\Theta_\varepsilon$ on $\End(E)$. We will mostly be interested in the case where $\alpha$ and $\beta$ are $(0,1)$-forms. In this case,
$$
\Omega_\varepsilon^D(\alpha,\beta) = \int_X \tr(\alpha \wedge \beta^*) \wedge \Theta_\varepsilon.
$$
The set of (not necessarily integrable) Dolbeault operators is an affine space with associated vector space $\Omega^{0,1}(X,\End(E))$. Thus it is in particular an infinite dimensional manifold and $\Omega_\varepsilon^D$ is formally a Kähler form on it. Moreover, by \cite{Atiyah_Bott,Donaldson}, the unitary gauge group has a Hamiltonian action on this manifold and the associated equivariant moment map is,
$$
\nu_{\infty,\varepsilon} : \dbar \mapsto \Lambda_\varepsilon F_{\dbar} + \i c_\varepsilon\Id_E.
$$
Concretely, it means that for all $\dbar$, for all $v$ tangent to $\dbar$ and all $a \in \Lie(\G^\C) = \Omega^0(X,\End(E))$,
\begin{equation}\label{EQ:Application moment dimension infinie}
    \scal{d\nu_{\infty,\varepsilon}(\dbar)v}{a}_\varepsilon = \Omega_\varepsilon^D(L_{\dbar}a,v),
\end{equation}
where $\dbar \mapsto L_{\dbar}a = \frac{\partial}{\partial t}|_{t = 0}\e^{ta} \cdot \dbar$ is the vector field induced by the infinitesimal action of $a$. We use here the same notation as in \cite{GRS}. We want now to reduce ourselves to the same kind of moment map, but for finite dimensional Lie groups. For this, we use first the same method as in \cite{DMS,Clarke_Tipler,Gabor} involving the Kuranishi slice.

Let $V$ be the finite dimensional complex space $H^{0,1}(X,\End(\Gr(\E)),\Theta_0)$ of harmonic $(0,1)$-forms with values in $\End(\Gr(\E))$ with respect to $\Theta_0$. Let,
$$
\Phi : B \rightarrow \Omega^{0,1}(X,\End(\Gr(\E))),
$$
be the Kuranishi slice \cite{Kuranishi} where $B \subset V$ is a ball around $0$. It is given as the local inverse around $0$ of
$$
\alpha \mapsto \alpha + \dbar_0^*\mathrm{Green}(\alpha \wedge \alpha),
$$
where $\mathrm{Green}$ is the Green operator associated to $\Delta_0 = \Delta_{\Theta_0,\nabla_0}$. In particular, it is a holomorphic embedding when $B$ and $\Omega^{0,1}(X,\End(E))$ are endowed with their usual complex structures. When $b \in B$, we set the Dolbeault operator,
$$
\dbar_b = \dbar_0 + \Phi(b).
$$
Recall that $\Phi$ verifies the following properties,
\begin{enumerate}
    \item $\Phi(0) = 0$ and $d\Phi(0) : v \mapsto v$ for $v \in T_0V = V \subset \Omega^{0,1}(X,\End(E))$.
    \item $B$ is $K$-invariant and $\Phi$ is $G$-equivariant where it is defined in the sense that for all $b \in B$ and all $g \in G$ such that $g \cdot b \in B$, $\dbar_{g \cdot b} = g \cdot \dbar_b$.
    \item For any small enough deformation $\dbar$ of $\dbar_0$, there is a gauge transformation $f \in \G^\C$ such that $f \cdot \dbar \in \dbar_0 + \Phi(B)$.
    \item $Z = \left\{b \in B\middle|\dbar_b^2 = 0\right\}$ is a complex subspace of $B$.
\end{enumerate}

Since $\E$ is a small holomorphic deformation of $\Gr(\E)$ (see for example \cite[Section 3]{Clarke_Tipler}), we can find some $b_0 \in B$ such that,
$$
\dbar_{b_0} = \dbar_0 + \Phi(b_0),
$$
is gauge equivalent to $\dbar_E$. From now on, we call,
$$
\mathcal{O} = G \cdot b_0 \cap B,
$$
the $G$-orbit of $b_0$ in $B$. Any $b \in \mathcal{O}$ is such that $\dbar_b$ is gauge equivalent to $\dbar_E$.

Then, when $\Theta$ is a balanced metric close to $\Theta_0$ such that $\E$ is $[\Theta]$-stable, we can consider two methods to find a HYM Dolbeault operator $\dbar$ in the gauge orbit of $\dbar_E$. The first one is to look for a point $b \in \mathcal{O}$ such that $\mu_{\infty,\varepsilon}(\dbar_b)$ is orthogonal to $\frak{k}$. It can be made using GIT because the projection of $\mu_{\infty,\varepsilon}(\dbar_b)$ onto $\frak{k}$ is a moment map on $B$ for a suitable $K$-invariant Kähler form. Then, we find a sequence of gauge transform $(f_k)$ such that $\mu_{\infty,\varepsilon}(f_k \cdot \dbar_b)$ approaches $0$ and we conclude thanks to a quantitative version of the inverse function theorem. This method is used in \cite{Sektnan_Tipler} to solve the HYM equation on a blow-up, in \cite{DMS} in the context of $Z$-critical equations and in \cite{Ortu_Sektnan} in the context of cscK equations.

The second method, which is used here, consists in perturbing first the moment map. We find a family $(\dbar_{\varepsilon,b})$ which varies smoothly with $b$ and $\varepsilon$ such that $\dbar_{\varepsilon,b}$ is in the gauge orbit of $\dbar_b$ and $\mu_{\infty,\varepsilon}(\dbar_{\varepsilon,b})$ lies in $\frak{k}$. In this case, we can see $b \mapsto \mu_{\infty,\varepsilon}(\dbar_{\varepsilon,b})$ as a moment map on $B$ with respect to a suitable $K$-invariant \textit{symplectic} form on it. When $b \in \mathcal{O}$ is a zero of this moment map, the operator $\dbar_{\varepsilon,b}$ is $\Theta_\varepsilon$-HYM and is in the gauge orbit of $\dbar_E$. This method is used in \cite{Gabor} for cscK equations, in \cite{Clarke_Tipler} to study wall-crossing phenomena for the HYM equation and in \cite{Ortu} for optimal symplectic connections. The main default of this method is that the symplectic form on $B$ with respect to which $b \mapsto \mu_{\infty,\varepsilon}(\dbar_{\varepsilon,b})$ is a moment map is not compatible with the complex structure in $B$. Therefore, we need to adapt the usual GIT methods to this particular setting.

Then, we deform $\Phi$ because we are interested in Dolbeault operators $\dbar$ such that $\nu_{\infty,\varepsilon}(\dbar)$ belongs to the Lie algebra $\frak{k}$ of $K$ as in \cite[Proposition 3.2]{Clarke_Tipler}. In the following, we build functions that depend on $\varepsilon \in U$. To understand accurately the continuity of these objects with respect to $\varepsilon$, let us introduce norms on $\Omega^{n - 1,n - 1}(X,\C)$. Let $\norme{\cdot}_A$ be a $\mC^k$ norm, a Hölder norm or a Sobolev norm that dominates the $\mC^1$ norm on the space of smooth forms with values in any smooth complex bundle $F$ on $X$. By the Sobolev injection theorems, it means that,
$$
A = \left\{
\begin{array}{lcl}
    \mC^k & \textrm{with} & k \geq 1,\\
    \mC^{k,\alpha} & \textrm{with} & k \geq 1, 0 < \alpha \leq 1,\\
    L_k^p & \textrm{with} & k - \frac{2n}{p} > 1.
\end{array}
\right.
$$
All of these norms can be computed thanks to a background metric on $X$ and a background metric on $F$. $\norme{\cdot}_A$ depends on the choice of these metrics but it is unique up to equivalence. We set,
$$
A' = \left\{
\begin{array}{lcl}
    \mC^{k - 1} & \textrm{if} & A = \mC^k,\\
    \mC^{k - 1,\alpha} & \textrm{if} & A = \mC^{k,\alpha},\\
    L_{k - 1}^p & \textrm{if} & A = L_k^p.
\end{array}
\right.\qquad
{'\!A} = \left\{
\begin{array}{lcl}
    \mC^{k + 1} & \textrm{if} & A = \mC^k,\\
    \mC^{k + 1,\alpha} & \textrm{if} & A = \mC^{k,\alpha},\\
    L_{k + 1}^p & \textrm{if} & A = L_k^p.
\end{array}
\right.
$$
Since $\norme{\cdot}_A$ dominates the $\mC^1$ norm, $\norme{\cdot}_{A'}$ dominates the $\mC^0$ norm and $\norme{\cdot}_{{'\!A}}$ dominates the $\mC^2$ norm.

\begin{lemma}\label{LEM:Applications lisses}
    The maps,
    $$
    \fonction{(\Omega^0(X,\End_H(E,h)),\norme{\cdot}_{{'\!A}}) \times B}{(\Omega^{0,1}(X,\End(E)),\norme{\cdot}_A)}{(s,b)}{\e^s \cdot \dbar_b - \dbar_0}
    $$
    $$
    \fonction{(U,\norme{\cdot}_{A'}) \times (\Omega^{0,1}(X,\End(E)),\norme{\cdot}_A)}{(\Omega^{1,1}(X,\End(E)),\norme{\cdot}_{A'})}{(\varepsilon,\alpha)}{\Lambda_\varepsilon F_{\dbar_0 + \alpha}}
    $$
    are smooth.
\end{lemma}
\begin{proof}
For the first one, we can use (\ref{EQ:Action de gauge}) to obtain,
$$
\e^s \cdot \dbar_b - \dbar_0 = \e^s \cdot \dbar_0 + \e^s\Phi(b)\e^{-s} - \dbar_0 = \e^s\dbar_0(\e^{-s}) + \e^s\Phi(b)\e^{-s}.
$$
The result follows from the smoothness of $\Phi$ and the fact that products are smooth for the ${'\!A}$ norm. See \cite[Theorem 4.12]{Adams_Fournier} for reference. For the second map, we have,
$$
F_{\dbar_0 + \alpha} = (\nabla_0 + \alpha - \alpha^*) \circ (\nabla_0 + \alpha - \alpha^*) = F_{\dbar_0} + \nabla_0(\alpha - \alpha^*) + (\alpha - \alpha^*) \wedge (\alpha - \alpha^*).
$$
Once more, the smoothness follows from the smoothness of products for the $A$ norm.
\end{proof}

From now on, when there is no ambiguity in the norms and the asymptotic developments involved, we use the $\mathrm{o}$ and $\mathrm{O}$ notations
$$
f(x) = g(x) + \mathrm{o}(x), \qquad \left(\textrm{resp. } f(x) = g(x) + \mathrm{O}(x)\right),
$$
to mean that
$$
\norme{f(x) - g(x)} = \mathrm{o}(\norme{x}), \qquad \left(\textrm{resp. } \norme{f(x) - g(x)} + \mathrm{O}(\norme{x})\right).
$$

\begin{proposition}\label{PRO:Tranche déformée}
    Let $\Omega_0^0(X,\End_H(E,h)) \subset \Omega^0(X,\End_H(E,h))$ be the space of all smooth Hermitian sections $s$ of $\End(E)$ such that,
    $$
    \int_X \tr(s)\Vol_0 = 0.
    $$
    Up to shrinking $U$ and $B$, there exists a smooth map,
    $$
    \sigma : (U,\norme{\cdot}_{\mC^0}) \times B \rightarrow (\Omega^0_0(X,\End_H(E,h)),\norme{\cdot}_{\mC^2}),
    $$
    that verifies $\sigma(0,0) = 0$, $\frac{\partial}{\partial b}|_{b = 0}\sigma(0,b) = 0$ and if we set,
    $$
    \tilde{\Phi} : \fonction{U \times B}{\Omega^{0,1}(X,\End(E))}{(\varepsilon,b)}{\e^{\sigma(\varepsilon,b)} \cdot \dbar_b - \dbar_0},
    $$
    $$
    \dbar_{\varepsilon,b} = \e^{\sigma(\varepsilon,b)} \cdot \dbar_b = \dbar_0 + \tilde{\Phi}(\varepsilon,b), \qquad F_{\varepsilon,b} = F_{\dbar_{\varepsilon,b}},
    $$
    we have,
    \begin{enumerate}
        \item For all $b \in \mathcal{O}$, $\dbar_{\varepsilon,b}$ is gauge equivalent to $\dbar_E$.
        \item For all $\varepsilon \in U$, $\tilde{\Phi}(\varepsilon,\cdot)$ is $K$-equivariant.
        \item For all $(\varepsilon,b) \in U \times B$, $\Lambda_\varepsilon F_{\varepsilon,b} + \i c_\varepsilon\Id_E \in \frak{k}$.
        \item $\frac{\partial}{\partial b}|_{b = 0}\tilde{\Phi}(0,b) : v \mapsto v$. In particular, up to shrinking $U \times B$, each $\tilde{\Phi}(\varepsilon,\cdot)$ is a smooth embedding.
        \item In a neighbourhood $U' \times B'$ of $0$ in $(U,\norme{\cdot}_{A'}) \times B$, the map,
        $$
        \sigma : U' \times B' \rightarrow (\Omega_0^0(X,\End_H(E,h)),\norme{\cdot}_{'\!A}),
        $$
        is smooth.
    \end{enumerate}
\end{proposition}
\begin{proof}
When $W$ is a space of functions, we denote by $A(W)$ its completion with respect to an $A$ norm.

Let $\Pi_\bot : \Omega_0^0(X,\End_H(E,h)) \rightarrow \i\frak{k}^\bot$ be the orthogonal projection on $\i\frak{k}^\bot$ with respect to $\scal{\cdot}{\cdot}_0$ and,
$$
\Psi : \fonction{U \times B \times \i\frak{k}^\bot}{\i\frak{k}^\bot}{(\varepsilon,b,s)}{\Pi_\bot(\Lambda_\varepsilon\i F_{\e^s \cdot \dbar_b} - c_\varepsilon\Id_E)}.
$$
By Lemma \ref{LEM:Applications lisses}, $\Psi$ is smooth when we endow the starting $\i\frak{k}^\bot$ with a $\mC^2$ norm, the arrival one with a $\mC^0$ norm and $U$ with a $\mC^0$ norm. And by Proposition \ref{PRO:Calcul dPsi(0)},
$$
\frac{\partial}{\partial s}_{|s = 0}\Psi(0,0,s) = \Pi_\bot\Delta_0,
$$
and $\Delta_0$ is a symmetric semi-definite positive operator and it is continuous from $\mC^2(\Omega_0^0(X,\End_H(E,h)))$ to $\mC^0(\Omega_0^0(X,\End_H(E,h)))$. According to Proposition \ref{PRO:Noyau Laplacien}, the kernel of the Laplacian $\Delta_0$ (thus its cokernel) is $\i\frak{k}$. In particular, $\Pi_\bot\Delta_0 = \Delta_0$ and,
$$
\frac{\partial}{\partial s}_{|s = 0}\Psi(0,0,s) = \Delta_0 : \mC^2(\i\frak{k}^\bot) \rightarrow \mC^0(\i\frak{k}^\bot),
$$
is an isomorphism.

By the implicit function theorem, up to shrinking $U \times B$, there is a unique smooth $\sigma : \mC^0(U) \times B \rightarrow \mC^2(\i\frak{k}^\bot)$ such that $\sigma(0,0) = 0$ and for all $(\varepsilon,b) \in \mC^0(U) \times B$, $\Psi(\varepsilon,b,\sigma(\varepsilon,b)) = 0$. When we set $\tilde{\Phi}(\varepsilon,b) = \e^{\sigma(\varepsilon,b)} \cdot \dbar_b - \dbar_0$, the first point is immediate and the third point is verified. For the second one, notice that for all $u \in K$ and all $(\varepsilon,b) \in U \times B$
\begin{align*}
    \Psi(\varepsilon,b,u^*\sigma(\varepsilon,u \cdot b)u) & = \Pi_\bot(\Lambda_\varepsilon\i F_{\e^{u^*\sigma(\varepsilon,u \cdot b)u} \cdot \dbar_b} - c_\varepsilon\Id_E)\\
    & = \Pi_\bot(\Lambda_\varepsilon\i F_{u^*\e^{\sigma(\varepsilon,u \cdot b)} \cdot \dbar_{u \cdot b}u} - c_\varepsilon\Id_E)\\
    & = \Pi_\bot(u^*(\Lambda_\varepsilon\i F_{\e^{\sigma(\varepsilon,u \cdot b)} \cdot \dbar_{u \cdot b}} - c_\varepsilon\Id_E)u) \textrm{ by (\ref{EQ:Action gauge unitaire}),}\\
    & = 0.
\end{align*}
By uniqueness of $\sigma$, it implies that $u^*\sigma(\varepsilon,u \cdot b)u = \sigma(\varepsilon,b)$. Therefore,
$$
\dbar_{\varepsilon,u \cdot b} = \dbar_0 + \tilde{\Phi}(\varepsilon,u \cdot b) = \e^{\sigma(\varepsilon,u \cdot b)} \cdot \dbar_{u \cdot b} = u\e^{\sigma(\varepsilon,b)}u^*u\dbar_bu^* = u(\dbar_0 + \tilde{\Phi}(\varepsilon,b))u^* = u\dbar_{\varepsilon,b}u^*.
$$
It proves the second point. For the fourth point, we have for all $b$,
\begin{align*}
    \Psi(0,b,0) & = \Pi_\bot(\Lambda_0\i F_{\dbar_b} - c_0\Id_E)\\
    & = \Pi_\bot(\Lambda_0\i F_{\dbar_0} + \Lambda_0\i\nabla_0(\Phi(b) - \Phi(b)^*) + \Lambda_0\i(\Phi(b) - \Phi(b)^*) \wedge (\Phi(b) - \Phi(b)^*) - c_0\Id_E)\\
    & = \Pi_\bot(\partial_0^*\Phi(b) + \dbar_0^*\Phi(b)^*) + \mathrm{o}(b) \textrm{ by (\ref{EQ:dbar_0 HYM}) and Lemma \ref{LEM:Identités de Kähler}.}
\end{align*}
Therefore, $\frac{\partial}{\partial b}_{|b = 0}\Psi(0,b,0) = \Pi_\bot(\partial_0^*d\Phi(0) + \dbar_0^*d\Phi(0)^*)$. Now, recall that the image of $d\Phi(0)$ is the space $V$ of $(0,1)$-harmonic forms. In particular, it is included in the kernel of $\partial_0^*$. Similarly, $d\Phi(0)^*$ takes values in $\ker(\dbar_0^*)$ because $\nabla_0$ is unitary. Thus $\frac{\partial}{\partial b}_{|b = 0}\Psi(0,b,0) = 0$. Therefore, when we differentiate the equality $\Psi(0,b,\sigma(0,b))$ at $b = 0$,
$$
0 = \frac{\partial}{\partial b}|_{b = 0}\Psi(0,b,0) + \frac{\partial}{\partial s}|_{s = 0}\Psi(0,0,s)\frac{\partial}{\partial b}|_{b = 0}\sigma(0,b) = \Delta_0\frac{\partial}{\partial b}|_{b = 0}\sigma(0,b).
$$
Since $\frac{\partial}{\partial b}|_{b = 0}\sigma(0,b)$ belongs to $\i\frak{k}^\bot$ where $\Delta_0$ is injective, we have $\frac{\partial}{\partial b}|_{b = 0}\sigma(0,b) = 0$ hence,
\begin{align*}
    \tilde{\Phi}(0,b) & = \e^{\sigma(0,b)} \cdot \dbar_b - \dbar_0\\
    & = \e^{\sigma(0,b)} \cdot \dbar_0 + \e^{\sigma(0,b)}\Phi(b)\e^{-\sigma(0,b)} - \dbar_0\\
    & = \e^{\sigma(0,b)}\dbar_0(\e^{-\sigma(0,b)}) + \e^{\sigma(0,b)}\Phi(b)\e^{-\sigma(0,b)} \textrm{ by (\ref{EQ:Action de gauge}),}\\
    & = \Phi(b) + \mathrm{o}(b).
\end{align*}
We deduce the fourth point and we showed along the way that $\frac{\partial}{\partial b}|_{b = 0}\sigma(0,b) = 0$. All we have left to show is that the $\sigma(\varepsilon,b)$ are smooth as sections when $\varepsilon \in U$ and that $\sigma$ has the wanted regularity. The first part follows from elliptic regularity. For the second one, we know by Lemma \ref{LEM:Applications lisses} that,
$$
\frac{\partial}{\partial s}|_{s = 0}\Psi(0,0,s) = \Delta_0 : (\i\frak{k},\norme{\cdot}_{'\!A}) \rightarrow (\i\frak{k},\norme{\cdot}_{A'}),
$$
is an isomorphism. Therefore, we can apply the implicit function theorem with these norms and we obtain a smooth map $\tilde{\sigma} : A'(U') \times B' \rightarrow {'\!A}(\i\frak{k}^\bot)$ which verifies $\Psi(\varepsilon,b,\tilde{\sigma}(\varepsilon,b)) = 0$.

By uniqueness of $\sigma$, $\sigma_{|U'} = \tilde{\sigma}$ is smooth from $(U',\norme{\cdot}_{A'}) \times B'$ to $(\i\frak{k}^\bot,\norme{\cdot}_{'\!A})$.
\end{proof}

Following \cite[Section 3.2]{Clarke_Tipler}, we define for all $\varepsilon \in U$, the closed $2$-form $\Omega_\varepsilon = \tilde{\Phi}(\varepsilon,\cdot)^*\Omega_\varepsilon^D$ on $B$ and,
$$
\nu_\varepsilon : \fonction{B}{\frak{k}}{b}{\nu_{\infty,\varepsilon}(\dbar_{\varepsilon,b}) = \Lambda_\varepsilon F_{\varepsilon,b} + \i c_\varepsilon\Id_E}.
$$
By Proposition \ref{PRO:Tranche déformée}, this function indeed takes values in $\frak{k}$. It is easy to compute that since $\nu_{\infty,\varepsilon}$ is a $K$-equivariant moment map for $\Omega_\varepsilon^D$ and $\tilde{\Phi}(\varepsilon,\cdot)$ is a $K$-equivariant embedding, $\nu_\varepsilon$ is a $K$-equivariant moment map for $\Omega_\varepsilon$, which is a symplectic form. We must be careful about the fact that $\Omega_\varepsilon$ is not compatible with the complex structures because $\tilde{\Phi}(\varepsilon,\cdot)$ is \textit{a priori} not holomorphic, but we still have the following positivity condition.
\begin{lemma}\label{LEM:Omega positive}
    Up to shrinking $U$ and $B$, for all $b \in B$ and $v \in T_bB\backslash\{0\}$, $\Omega_\varepsilon(v,\i v)$ is positive.
\end{lemma}
\begin{proof}
This is due to the fact that $\frac{\partial}{\partial b}|_{b = 0}\tilde{\Phi}(0,b) : v \mapsto v$ is $\C$-linear, hence, for all $v \in T_0B = V$,
$$
\Omega_0(v,\i v) = \Omega_0^D\left(\frac{\partial}{\partial b}|_{b = 0}\tilde{\Phi}(0,b)v,\frac{\partial}{\partial b}|_{b = 0}\tilde{\Phi}(0,b)(\i v)\right) = \Omega_0^D(v,\i v) > 0.
$$
It means that $\Omega_0(\cdot,\i \cdot)$ at $b = 0$ is a positive quadratic form. By openness of positivity, up to shrinking $U$ and $B$, $\Omega_\varepsilon$ is positive at any point $b$ of $B$, hence the result.
\end{proof}

\subsection{Bound on the norm of $b$}

Now, when $\Theta_\varepsilon$ is in the stable cone of $\E$, we want to find a $\Theta_\varepsilon$-HYM connection $\dbar_{\varepsilon,b}$ in the image of $\dbar_0 + \tilde{\Phi}(\varepsilon,\cdot)$ restricted to the orbit $\mathcal{O}$ and then control the norm of $b$ in function of the norm of $\varepsilon$. Finally, knowing the regularity of $\tilde{\Phi}$ at the origin, we can control the norm of $\dbar_{\varepsilon,b}$ in function of $\varepsilon$.

The following proposition gives an estimate on the norm of $b$ in term of the norms of $\nu_\varepsilon(b)$ and $\varepsilon$. Notice that when $\nu_\varepsilon(b) = 0$, the norm of $b$ is controlled by $\varepsilon$ only. We denote by $\norme{\cdot}$ any Euclidean norms on the finite dimensional vector spaces $V$ and $H^{n - 1,n - 1}(X,\R)$. We recall that $\norme{\cdot}_\varepsilon$ is the $L^2$ norm on $\frak{k}$ with respect to $\Vol_\varepsilon$ and $\norme{\cdot}_{\mC^0}$ is a $\mC^0$-norm on $\Omega^{n - 1,n - 1}(X,\R)$.

\begin{proposition}\label{PRO:Inégalité norme b}
    Up to shrinking $U$ and $B$, for all $\varepsilon \in U$ and $b \in \overline{\mathcal{O}}$, $\norme{b}^2 \leq C(\norme{\nu_\varepsilon(b)}_\varepsilon + \norme{\varepsilon}_{\mC^0}^2 + \norme{[\varepsilon]})$ for some positive constant $C$ independent from $\varepsilon$ and $b$.
\end{proposition}
\begin{proof}
By density, it is enough to verify it for $b \in \mathcal{O}$.

\vspace{1mm}
\noindent\textbf{Step 1 :} Finding an orthogonal decomposition of $E$ where $\Phi(b)$ is strictly upper triangular.
\vspace{1mm}

Recall from the beginning of Section \ref{SEC:Résulats analytiques} that,
$$
\dbar_E - \dbar_0 = \sum_{1 \leq i < j \leq m} \gamma_{ij},
$$
where $\gamma_{ij} \in \Omega^{0,1}(X,\Hom(G_j,G_i))$. If $t > 0$ and $g_t = \sum_{k = 1}^m t^k\Id_{G_k}$, we have $g_t \cdot \dbar_0 = \dbar_0$ so,
$$
g_t \cdot \dbar_E - \dbar_0 = g_t(\dbar_E - \dbar_0)g_t^{-1} = \sum_{1 \leq i < j \leq m} t^{i - j}\gamma_{ij} \tend{t}{+\infty} 0.
$$
In other words, we can find Dolbeault operators $\dbar$ in the gauge orbit of $\dbar_E$ arbitrarily close to $\dbar_0$. By property of the Kuranishi slice, we can find elements $b' \in \mathcal{O}$ arbitrarily close to $0$. By Hilbert--Mumford criterion \cite[Theorem 12.4]{GRS}, there is a $\xi \in \i\frak{k}$ such that,
$$
\e^{t\xi} \cdot b \tend{t}{+\infty} 0.
$$
Since $\Phi$ is $G$-equivariant, $\e^{t\xi} \cdot \Phi(b) \tend{t}{+\infty} 0$. $\xi$ is a Hermitian holomorphic global section of $\End(\E_0)$. By the spectral theorem, we can can diagonalise it as,
$$
\xi = \sum_{k = 1}^p \lambda_k\Id_{G_{b,k}},
$$
where the $\lambda_1 < \cdots < \lambda_p$ are its increasing eigenvalues and the $G_{b,k}$ are its eigenspaces. They are sub-bundles of $\E_0$ and they are pairwise orthogonal. Since $\E_0$ is $[\Theta_0]$-polystable, it implies that the $[\Theta_0]$-slope of each $G_{b,k}$ equals the $[\Theta_0]$-slope of $\E_0$. In the orthogonal decomposition,
$$
\Omega^{0,1}(X,\End(E)) = \bigoplus_{1 \leq i,j \leq p} \Omega^{0,1}(X,\Hom(G_{b,j},G_{b,i})),
$$
we can decompose $\Phi(b)$ as,
$$
\Phi(b) = \sum_{1 \leq i,j \leq p} \gamma_{ij,b}.
$$
We have,
$$
\e^{t\xi} \cdot \Phi(b) = \sum_{1 \leq i,j \leq p} \e^{(\lambda_i - \lambda_j)t}\gamma_{ij,b} \tend{t}{+\infty} 0.
$$
We deduce that for all $i \geq j$, $\gamma_{ij,b} = 0$ so $\Phi(b)$ is strictly upper triangular. In particular, if we set, $F_{b,k} = \bigoplus_{i = 1}^k G_{b,i}$, each $F_{k,b}$ is preserved by $\dbar_b$ so,
$$
\F_{k,b} = (F_{k,b},\dbar_{E|F_{k,b}}),
$$
is a holomorphic sub-bundle of $\E_b = (E,\dbar_b)$.

\vspace{1mm}
\noindent\textbf{Step 2 :} Computing the scalar product between $\nu_{\infty,\varepsilon}(b)$ and each $a_S = \frac{\i}{\rk(S)}\Id_S - \frac{\i}{\rk(S^\bot)}\Id_{S^\bot} \in \frak{k}$.
\vspace{1mm}

Let $\mathcal{S} \subset \E_b$ be a sub-bundle with $0 < \rk(\mathcal{S}) < \rk(\E)$ and $S$ its smooth structure. Let $\Id_S$ defined as the orthogonal projection on $S$ and $\Id_{S^\bot} = \Id_E - \Id_S$ the orthogonal projection on $S^\bot$. Then, in the decomposition $E = S \oplus S^\bot$, we can write $\dbar_b = \dbar_{b|S} + \beta_{S,b}$. $\beta_{S,b}$ is the second fundamental form of $S$. and by \cite[Equation (1.6.12)]{Kobayashi},
\begin{equation}\label{EQ:Matrice par blocks courbure}
    \i F_b = \begin{pmatrix} \i F_S - \i \beta_{S,b} \wedge \beta_{S,b}^* & \i\partial_b\beta_{S,b} \\ -\i\dbar_b\beta_{S,b}^* & \i F_{S^\bot} - \i\beta_{S,b}^* \wedge \beta_{S,b} \end{pmatrix},
\end{equation}
where $F_b$ is the curvature of the Chern connection associated to $\dbar_b$ on $E$, $F_S$ is some curvature form on $\mathcal{S}$ and $F_{S^\bot}$ is some curvature form on $\E_b/\mathcal{S}$ (they depend on $b$). By Chern-Weil theory,
\begin{equation}\label{EQ:Chern-Weil}
    \int_X \tr(\i F_S) \wedge \Theta_\varepsilon = 2\pi\rk(S)\mu_{[\varepsilon]}(S), \qquad \int_X \tr(\i F_{S^\bot}) \wedge \Theta_\varepsilon = 2\pi\rk(E/S)\mu_{[\varepsilon]}(E/S),
\end{equation}
and for any $(0,1)$-form $\beta$, $\displaystyle \int_X \tr(\i\beta^*\wedge\beta) \wedge \Theta_\varepsilon = \norme{\beta}_\varepsilon^2$ and $\displaystyle \int_X \tr(\i\beta\wedge\beta^*) \wedge \Theta_\varepsilon = -\norme{\beta}_\varepsilon^2$. We can now compute that,
\begin{align*}
    \scal{\nu_{\infty,\varepsilon}(\dbar_b)}{a_S}_\varepsilon & = \scal{\Lambda_\varepsilon F_b + \i c_\varepsilon\Id_E}{a_S}_\varepsilon\\
    & = \int_X \tr\left(\Lambda_\varepsilon\i F_b\left(\frac{1}{\rk(S)}\Id_S - \frac{1}{\rk(S^\bot)}\Id_{S^\bot}\right)\right)\Vol_\varepsilon\\
    & = \frac{1}{\rk(S)}\int_X \tr(\i F_b\Id_S) \wedge \Theta_\varepsilon - \frac{1}{\rk(S^\bot)}\int_X \tr(\i F_b\Id_{S^\bot}) \wedge \Theta_\varepsilon\\
    & = \frac{1}{\rk(S)}\int_X \tr(\i F_S - \i \beta_{S,b}\wedge\beta_{S,b}^*) \wedge \Theta_\varepsilon - \frac{1}{\rk(S^\bot)}\int_X \tr(\i F_{S^\bot} - \i\beta_{S,b}^*\wedge\beta_{S,b}\Id_{S^\bot}) \wedge \Theta_\varepsilon \textrm{ by (\ref{EQ:Matrice par blocks courbure}),}\\
    & = 2\pi\mu_{[\varepsilon]}(S) + \frac{1}{\rk(S)}\norme{\beta_{S,b}}_\varepsilon^2 - 2\pi\mu_{[\varepsilon]}(E/S) + \frac{1}{\rk(S^\bot)}\norme{\beta_{S,b}}_\varepsilon^2 \textrm{ by (\ref{EQ:Chern-Weil}),}\\
    & = -2\pi l_{\mathcal{S}}([\varepsilon]) + \left(\frac{1}{\rk(S)} + \frac{1}{\rk(S^\bot)}\right)\norme{\beta_{S,b}}_\varepsilon^2.
\end{align*}

\vspace{1mm}
\noindent\textbf{Step 3 :} Inequality with $\nu_{\infty,\varepsilon}$.
\vspace{1mm}

Applying step 2 with the particular case $\mathcal{S} = \F_{k,b} \subset \E$ with $k < p$, we obtain,
$$
\beta_{S,b} = \Id_{F_{k,b}}\dbar_{b|F_{k,b}^\bot} = \sum_{i \leq k \leq j} \gamma_{ij,b}.
$$
Moreover, all the $l_{\mathcal{S}} : H^{n - 1,n - 1}(X,\R) \rightarrow \R$ are continuous thus,
\begin{align}
    \sum_{k = 1}^p \scal{\nu_{\infty,\varepsilon}(\dbar_b)}{a_{F_{k,b}}}_\varepsilon & = \sum_{k = 1}^p -2\pi l_{\F_{k,b}}([\varepsilon]) + \left(\frac{1}{\rk(F_{k,b})} + \frac{1}{\rk(F_{k,b}^\bot)}\right)\norme{\sum_{i \leq k \leq j} \gamma_{ij,b}}_\varepsilon^2 \nonumber\\
    & \geq \sum_{k = 1}^p C_3\sum_{i \leq k \leq j} \norme{\gamma_{ij,b}}_0^2 - C_4\norme{[\varepsilon]} \nonumber\\
    & \geq C_3\norme{\Phi(b)}_0^2 - C_4\norme{[\varepsilon]} \nonumber\\
    & \geq \frac{C_3}{C_1}\norme{b}^2 - C_4\norme{[\varepsilon]}, \label{EQ:Borne 1}
\end{align}
where $C_3$ and $C_4$ are positive constants that only depend on $\rk(\E)$, $\max_{1 \leq k \leq p}\left\{\norme{l_{\F_{k,b}}}\right\}$,\\
$\sup_{\varepsilon \in U}\{\Vol_\varepsilon(X)\}$ (which is finite up to shrinking $U$) and the fact that the $\norme{\cdot}_\varepsilon$ are uniformly equivalent. Clearly, it exists a constant $C_5$ independent from $\varepsilon$ and $b$ such that for all $k$ and $\varepsilon$, $\norme{a_{F_{k,b}}}_\varepsilon \leq C_5$. Let for all $\varepsilon$, $\Pi_\varepsilon$ be the orthogonal projection on $\frak{k}$ with respect to $\scal{\cdot}{\cdot}_\varepsilon$. We have,
\begin{align}
    \sum_{k = 1}^p \scal{\nu_{\infty,\varepsilon}(\dbar_b)}{a_{F_{k,b}}}_\varepsilon & = \sum_{k = 1}^p \scal{\Pi_\varepsilon\nu_{\infty,\varepsilon}(\dbar_b)}{a_{F_{k,b}}}_\varepsilon \nonumber\\
    & \leq C_5\sum_{k = 1}^p \norme{\Pi_\varepsilon\nu_{\infty,\varepsilon}(\dbar_b)}_\varepsilon \nonumber\\
    & = pC_5\norme{\Pi_\varepsilon\nu_{\infty,\varepsilon}(\dbar_b)}. \label{EQ:Borne 2}
\end{align}
We deduce from (\ref{EQ:Borne 1}) and (\ref{EQ:Borne 2}) that $\norme{b}^2 = \mathrm{O}\!\left(\norme{\Pi_\varepsilon\nu_{\infty,\varepsilon}(\dbar_b)}_\varepsilon + \norme{[\varepsilon]}\right)$ where the $\mathrm{O}$ is independent from $\varepsilon$ and $b$. From now on, all the asymptotic expansions are with respect to $(\norme{\varepsilon}_{\mC^0},b) \rightarrow (0,0)$.

\vspace{1mm}
\noindent\textbf{Step 4 :} Conclusion.
\vspace{1mm}

By Proposition \ref{PRO:Tranche déformée}, we have $\nu_\varepsilon(b) = \nu_{\infty,\varepsilon}(\dbar_{\varepsilon,b})$ with $\dbar_{\varepsilon,b} = \e^{\sigma(\varepsilon,b)} \cdot \dbar_b$ for some smooth,
$$
\sigma : \mC^0(U) \times B \rightarrow \mC^2(\Omega_0^0(X,\End_H(E,h))),
$$
which verifies $\sigma(0,0) = 0$ and $\frac{\partial}{\partial b}|_{b = 0}\sigma(0,b) = 0$. Therefore,
\begin{equation}\label{EQ:Borne sigma}
    \norme{\sigma(\varepsilon,b)}_{\mC^2} = \mathrm{O}(\norme{\varepsilon}_{\mC^0} + \norme{b}^2).
\end{equation}
Moreover, $\Delta_0$ takes values in $\frak{k}^\bot$ (for $\scal{\cdot}{\cdot}_0$) hence $\Pi_0\Delta_0 = 0$. Therefore, the map,
$$
\fonction{(U,\norme{\cdot}_{\mC^0}) \times B \times (\Omega^0(X,\End_H(E,h)),\norme{\cdot}_{\mC^2})}{\frak{k}}{(\varepsilon,b,s)}{\Pi_\varepsilon\Delta_{\varepsilon,\nabla_b}s},
$$
vanishes at each $(0,0,s)$. By Lemma \ref{LEM:Applications lisses}, it is smooth hence,
\begin{equation}\label{EQ:Borne Delta}
    \norme{\Pi_\varepsilon\Delta_{\varepsilon,\nabla_b}s}_\varepsilon = \mathrm{O}((\norme{\varepsilon}_{\mC^0} + \norme{b})\norme{s}_{\mC^2}).
\end{equation}
Using the fact that $\nu_\varepsilon$ takes values in $\frak{k}$,
\begin{align*}
    \norme{\nu_\varepsilon(b) - \Pi_\varepsilon\nu_{\infty,\varepsilon}(\dbar_b)}_\varepsilon & = \norme{\Pi_\varepsilon(\nu_\varepsilon(b) - \nu_{\infty,\varepsilon}(\dbar_b))}_\varepsilon\\
    & = \norme{\Pi_\varepsilon(\nu_{\infty,\varepsilon}(\e^{\sigma(\varepsilon,b)} \cdot \dbar_b) - \nu_{\infty,\varepsilon}(\dbar_b))}_\varepsilon\\
    & = \norme{\Pi_\varepsilon\Delta_{\varepsilon,\nabla_b}\sigma(\varepsilon,b)}_\varepsilon + \mathrm{O}(\norme{\sigma(\varepsilon,b)}_{\mC^2}^2) \textrm{ by Proposition \ref{PRO:Calcul dPsi(0)},}\\
    & = \mathrm{O}((\norme{\varepsilon}_{\mC^0} + \norme{b})\norme{\sigma(\varepsilon,b)}_{\mC^2}) + \mathrm{O}(\norme{\sigma(\varepsilon,b)}_{\mC^2}^2) \textrm{ by (\ref{EQ:Borne sigma}) and (\ref{EQ:Borne Delta}),}\\
    & = \mathrm{O}(\norme{\varepsilon}_{\mC^0}^2 + \norme{\varepsilon}_{\mC^0}\norme{b} + \norme{b}^3).
\end{align*}
Finally,
$$
\norme{b}^2 = \mathrm{O}\!\left(\norme{\Pi_\varepsilon\nu_{\infty,\varepsilon}(\dbar_b)}_\varepsilon + \norme{[\varepsilon]}\right) = \mathrm{O}(\norme{\nu_\varepsilon(b)}_\varepsilon + \norme{\varepsilon}_{\mC^0}^2 + \norme{\varepsilon}_{\mC^0}\norme{b} + \norme{b}^3 + \norme{[\varepsilon]}),
$$
so,
\begin{align*}
    \frac{\norme{b}^2}{\norme{\nu_\varepsilon(b)}_\varepsilon + \norme{\varepsilon}_{\mC^0}^2 + \norme{[\varepsilon]}} & = \mathrm{O}\!\left(\frac{\norme{\nu_\varepsilon(b)}_\varepsilon + \norme{\varepsilon}_{\mC^0}^2 + \norme{\varepsilon}_{\mC^0}\norme{b} + \norme{b}^3 + \norme{[\varepsilon]}}{\norme{\nu_\varepsilon(b)}_\varepsilon + \norme{\varepsilon}_{\mC^0}^2 + \norme{[\varepsilon]}}\right)\\
    & = \mathrm{O}\!\left(1 + \frac{\norme{\varepsilon}_{\mC^0}\norme{b}}{\norme{\nu_\varepsilon(b)}_\varepsilon + \norme{\varepsilon}_{\mC^0}^2 + \norme{[\varepsilon]}}\right) + \mathrm{o}\!\left(\frac{\norme{b}^2}{\norme{\nu_\varepsilon(b)}_\varepsilon + \norme{\varepsilon}_{\mC^0}^2 + \norme{[\varepsilon]}}\right)\\
    & = \mathrm{O}\!\left(1 + \frac{\norme{b}}{\sqrt{\norme{\nu_\varepsilon(b)}_\varepsilon + \norme{\varepsilon}_{\mC^0}^2 + \norme{[\varepsilon]}}}\right) + \mathrm{o}\!\left(\frac{\norme{b}^2}{\norme{\nu_\varepsilon(b)}_\varepsilon + \norme{\varepsilon}_{\mC^0}^2 + \norme{[\varepsilon]}}\right)
\end{align*}
\end{proof}
We deduce that $\frac{\norme{b}^2}{\norme{\nu_\varepsilon(b)}_\varepsilon + \norme{\varepsilon}_{\mC^0}^2 + \norme{[\varepsilon]}}$ is bounded hence the result.

\subsection{Finding a zero of the moment map}

The method now consists in defining a vector field whose flow converges toward a zero of the moment map. When $b \in B$ and $a \in \frak{g}$, we denote by $L_ba = \frac{\partial}{\partial t}|_{t = 0}\e^{ta} \cdot b \in T_bB$ the infinitesimal action of $a$ on $b$. Following \cite[Chapter 3]{GRS}, we define the vector field,
$$
V_\varepsilon : b \mapsto -\i L_b\nu_\varepsilon(b).
$$
This flow can be seen as a finite dimensional analogue of the Yang--Mills gradient flow introduced by Donaldson \cite{Donaldson} to find HYM connections. See also \cite[Chapter 6]{Kobayashi} for reference.

Let $t \mapsto b_{\varepsilon}(t)$ be the flow of $V_\varepsilon$ whose starting point will be defined by the following proposition.
\begin{proposition}\label{PRO:Flot défini}
    Up to shrinking $U$, for all $\varepsilon \in U$, if we choose $b_\varepsilon(0)$ close enough to $0$, $b_\varepsilon$ is defined for all non-negative times, $t \mapsto \norme{\nu_\varepsilon(b_\varepsilon(t))}_\varepsilon$ decreases and $\norme{b_\varepsilon(t)} = \mathrm{O}\!\left(\sqrt{\norme{\varepsilon}_{\mC^0}}\right)$ uniformly with respect to $t$.
\end{proposition}
\begin{proof}
For all $t$ where the flow is defined, whatever is the choice of the starting point,
\begin{align*}
    \frac{\partial}{\partial t}\left(\frac{1}{2}\norme{\nu_\varepsilon(b_\varepsilon(t))}_\varepsilon^2\right) & = \scal{d\nu_\varepsilon(b_\varepsilon(t))(-\i L_b\nu_\varepsilon(b_\varepsilon(t)))}{\nu_\varepsilon(b_\varepsilon(t))}_\varepsilon\\
    & = -\Omega_\varepsilon(L_b\nu_\varepsilon(b_\varepsilon(t)),\i L_b\nu_\varepsilon(b_\varepsilon(t))) \textrm{ by the moment map property,}\\
    & \leq 0 \textrm{ by Lemma \ref{LEM:Omega positive}}.
\end{align*}
It proves that $t \mapsto \norme{\nu_\varepsilon(b_\varepsilon(t))}_\varepsilon$ decreases while the flow is defined. In particular, by Proposition \ref{PRO:Inégalité norme b},
\begin{equation}\label{EQ:Inégalité norme flux}
    \norme{b_\varepsilon(t)}^2 \leq C(\norme{\nu_\varepsilon(b_\varepsilon(t))}_\varepsilon + \norme{\varepsilon}_{\mC^0}^2 + \norme{[\varepsilon]}) \leq C(\norme{\nu_\varepsilon(b_\varepsilon(0))}_\varepsilon + \norme{\varepsilon}_{\mC^0}^2 + \norme{[\varepsilon]})
\end{equation}
Let $r > 0$ such that the closed ball of centre $0$ and radius $r$ is included in $B$. Assume, up to shrinking $U$, that for all $\varepsilon \in U$,
$$
\norme{\varepsilon}_{\mC^0}^2 + \norme{[\varepsilon]} \leq \frac{r^2}{2C}, \qquad \norme{\nu_\varepsilon(0)}_\varepsilon \leq \frac{r^2}{4C}.
$$
Then choose $b_\varepsilon(0)$ close enough to $0$ so $\norme{\nu_\varepsilon(b_\varepsilon(0))}_\varepsilon \leq 2\norme{\nu_\varepsilon(0)}_\varepsilon \leq \frac{r^2}{2C}$. It implies by (\ref{EQ:Inégalité norme flux}) that $\norme{b_\varepsilon(t)} \leq r$ so $b_\varepsilon$ stays in a compact set included in $B$. By the finite time explosion theorem, the flow is defined at all time. Moreover,
$$
\norme{b_\varepsilon(t)}^2 \leq C(\norme{\nu_\varepsilon(b_\varepsilon(0))}_\varepsilon + \norme{\varepsilon}_{\mC^0}^2 + \norme{[\varepsilon]}) \leq C(2\norme{\nu_\varepsilon(0)}_\varepsilon + \norme{\varepsilon}_{\mC^0}^2 + \norme{[\varepsilon]}) = \mathrm{O}(\norme{\varepsilon}_{\mC^0}).
$$
\end{proof}

Let $t \mapsto g_\varepsilon(t)$ in $G$ be the smooth function defined by,
$$
g_\varepsilon(0) = \Id_E, \qquad g_\varepsilon(t)^{-1}g_\varepsilon'(t) = \i\nu_\varepsilon(b_\varepsilon(t)).
$$
Similarly to \cite[Lemma 3.2]{GRS}, we have for all $t$,
$$
\frac{\partial}{\partial t}(g_\varepsilon(t)^{-1} \cdot b_\varepsilon(0)) = -g_\varepsilon(t)^{-1}g_\varepsilon'(t)g_\varepsilon^{-1}(t) \cdot b_\varepsilon(0) = -\i L_{g_\varepsilon(t)^{-1} \cdot b_\varepsilon(0)}\nu_\varepsilon(b_\varepsilon(t)).
$$
By uniqueness of the flow, it proves that for all $t$, $g_\varepsilon(t)^{-1} \cdot b_\varepsilon(0) = b_\varepsilon(t)$. In particular, $\mathcal{O}$ is preserved by the flow. Let for all $t$, $\tilde{g}_\varepsilon(t) = \e^{\sigma(\varepsilon,b_\varepsilon(0))}g_\varepsilon(t)\e^{-\sigma(\varepsilon,b_\varepsilon(t))} \in \G^\C$. We have,
\begin{align}
    \dbar_{\varepsilon,b_\varepsilon(t)} & = \e^{\sigma(\varepsilon,b_\varepsilon(t))} \cdot \dbar_{b_\varepsilon(t)} \nonumber\\
    & = \e^{\sigma(\varepsilon,b_\varepsilon(t))}g_\varepsilon(t)^{-1} \cdot \dbar_{b_\varepsilon(0)} \nonumber\\
    & = \e^{\sigma(\varepsilon,b_\varepsilon(t))}g_\varepsilon(t)^{-1}\e^{-\sigma(\varepsilon,b_\varepsilon(0))} \cdot \dbar_{\varepsilon,b_\varepsilon(0)} \nonumber\\
    & = \tilde{g}_\varepsilon(t)^{-1} \cdot \dbar_{\varepsilon,b_\varepsilon(0)}. \label{EQ:dbar flux}
\end{align}
Now, when $\dbar$ is a Dolbeault operator on $E$, we denote by $M_{\varepsilon,\dbar}$ the Donaldson functional \cite[Proposition 6]{Donaldson} associated with $\Theta_\varepsilon$ and $\dbar$. It takes two Hermitian metrics on $E$ as argument and is characterised by,
\begin{equation}\label{EQ:Chasles}
    \forall k_1,k_2,k_3, M_{\varepsilon,\dbar}(k_1,k_2) + M_{\varepsilon,\dbar}(k_2,k_3) = M_{\varepsilon,\dbar}(k_1,k_3),
\end{equation}
\begin{equation}\label{EQ:Dérivée Donaldson}
    \frac{\partial}{\partial s}|_{s = 0}M_{\varepsilon,\dbar}(k,\e^{-s/2} \cdot k)v = \int_X \tr\left(v(\Lambda_\varepsilon\i F_{\dbar,k} - c_\varepsilon\Id_E)\right)\Vol_\varepsilon.
\end{equation}
Here, $s \in \Omega^0(X,\End_H(E,k))$ and $v \in T_s\Omega^0(X,\End_H(E,k))$ are $k$-Hermitian, $F_{\dbar,k}$ is the curvature associated to the Chern connection of $(\dbar,k)$ and $\G^\C$ acts on metrics by $(f \cdot k)(\xi,\eta) = k(f^{-1}\xi,f^{-1}\eta)$\footnote{In the literature, it is more common to only consider the action of a positive definite Hermitian $\e^s$ with respect to the metric $k$. In this case, the notation used is $k\e^s : (\xi,\eta) \mapsto k(\e^s\xi,\eta)$. By symmetry, we obtain that $\e^{-s/2} \cdot k = k\e^s$. To avoid confusion, we only the use in this case the notation $\e^{-s/2} \cdot k$.}. Donaldson and Simpson showed that on a stable bundle, $M(h,\e^{-s/2} \cdot h)$ uniformly bounds $s$ \cite{Donaldson,Simpson}. We refer the reader to Simpson's proof of this result \cite[Proposition 5.3]{Simpson}\footnote{Simpson proves it in the Kähler case but the closedness of $\omega$ is not used in the proof of \cite[Proposition 5.3]{Simpson}. Only the closedness of $\omega^{n - 1}$ is necessary for \cite[Proposition 5.1]{Simpson}. He also only proves it when $\tr(s) = 0$ identically, but this hypothesis is only useful to get that the $u_\infty$ built in \cite[Lemma 5.4]{Simpson} is not a homothety (see \cite[Lemma 5.5]{Simpson}). For this, it is enough to know that $\int_X \tr(s)\Vol_0 = 0$ so the inequality holds in this more general case.}. Concretely, when $(E,\dbar)$ is $[\Theta_\varepsilon]$-stable, there exists positive constants $C_{1,\varepsilon},C_{2,\varepsilon}$ that depend on $\varepsilon$ (and $\dbar$) such that for all $s \in \Omega_0^0(X,\End_H(E,h))$,
\begin{equation}\label{EQ:Simpson}
    \norme{s}_\varepsilon \leq C_{1,\varepsilon} + C_{2,\varepsilon}M_{\varepsilon,\dbar}(h,\e^{-s/2} \cdot h).
\end{equation}
We could even get this inequality with a $\mC^0$ norm instead of an $L^2$ norm on $s$. Now, let for all $b \in \mathcal{O}$,
$$
\varphi_\varepsilon : b \mapsto M_{\varepsilon,\dbar_E}(h,f_{\varepsilon,b} \cdot h),
$$
where $f_{\varepsilon,b}$ is defined as the gauge transformation such that $\dbar_{\varepsilon,b} = f_{\varepsilon,b}^{-1} \cdot \dbar_E$. It is unique up to a constant homothety by simplicity of $\E$. We make it unique by normalising,
$$
\int_X \ln(\det(f_{\varepsilon,b})) \, \Vol_0 = 0.
$$
We want now to show that $\varphi_\varepsilon$ decreases with the flow and use the inequality (\ref{EQ:Simpson}) to show that $g_\varepsilon$ is bounded, thus converges up to extraction of a sub-sequence.

In the classical case where we consider the heat equation flow, the derivative of $t \mapsto \varphi_\varepsilon(b_\varepsilon(t))$ is given by $-2\norme{\nu_\varepsilon(b_\varepsilon(t))}_\varepsilon^2$ (the constant $2$ may disappear in function of the conventions). See for reference \cite[Proposition 6.9.1]{Kobayashi}, \cite[Lemma 7.1]{Simpson}, \cite[Section 1.2]{Donaldson}. It is a natural consequence of (\ref{EQ:Dérivée Donaldson}). Here, the fact that $\tilde{g}_\varepsilon(t)$ is not exactly $g_\varepsilon(t)$ makes this equality false in general. However, since these two gauge transformations are close when $\varepsilon \rightarrow 0$ uniformly with respect to $t$, we are still able to show the wanted decrease. The issue is that we fix the metric and we make the complex structure vary to find a HYM connection. However, Donaldson's functional works better with variations of the metric when the holomorphic structure is fixed. This issue however, is only technical and we can move from one to the other. The proof of the next proposition is the only part of this article where we need to consider another metric on $E$.

\begin{proposition}\label{PRO:phi décroissante}
    We have for all $t$,
    $$
    \frac{\partial}{\partial t}\varphi_\varepsilon(b_\varepsilon(t)) = -2\norme{\nu_\varepsilon(b_\varepsilon(t))}_\varepsilon^2 + \mathrm{o}(\norme{\nu_\varepsilon(b_\varepsilon(t))}_\varepsilon^2),
    $$
    where the $\mathrm{o}$ is when $\varepsilon \rightarrow 0$ and is uniform with respect to $t$. In particular, up to shrinking $U$, for all $\varepsilon$, $t \mapsto \varphi_\varepsilon(b_\varepsilon(t))$ decreases.
\end{proposition}
\begin{proof}
$\varphi_\varepsilon$ is defined intrinsically in the sense that it doesn't depend on the choice of the starting point $b_\varepsilon(0)$. Since we have moreover a uniform bound on the norm of $b_\varepsilon(t)$ given by Proposition \ref{PRO:Flot défini}, it is enough to verify the wanted equality at $t = 0$ and the uniform bound follows.

\vspace{1mm}
\noindent\textbf{Step 1 :} Derivative of the Hermitian part of a gauge transformation.
\vspace{1mm}

For all $t$ near $0$, by (\ref{EQ:dbar flux}), we have $\dbar_{\varepsilon,b_\varepsilon(t)} = \tilde{g}_\varepsilon(t)^{-1} \cdot \dbar_{\varepsilon,b_\varepsilon(0)} = \tilde{g}_\varepsilon(t)^{-1}f_{\varepsilon,b_\varepsilon(0)}^{-1} \cdot \dbar_E$ hence $f_{\varepsilon,b_\varepsilon(t)} = f_{\varepsilon,b_\varepsilon(0)}\tilde{g}_\varepsilon(t)$. In particular, if we set $k = f_{\varepsilon,b_\varepsilon(0)} \cdot h$, we have,
$$
f_{\varepsilon,b_\varepsilon(t)} \cdot h = f_{\varepsilon,b_\varepsilon(0)}\tilde{g}_\varepsilon(t)f_{\varepsilon,b_\varepsilon(0)}^{-1} \cdot k.
$$

Let us write the polar decomposition of $f_{\varepsilon,b_\varepsilon(0)}\tilde{g}_\varepsilon(t)f_{\varepsilon,b_\varepsilon(0)}^{-1}$ with respect to $k$ as $\e^{-s(t)/2}u(t)$ where $u(t)$ is the unitary part (on the right) of $f_{\varepsilon,b_\varepsilon(0)}\tilde{g}_\varepsilon(t)f_{\varepsilon,b_\varepsilon(0)}^{-1}$. In particular,
$$
\e^{-s(t)} = f_{\varepsilon,b_\varepsilon(0)}\tilde{g}_\varepsilon(t)f_{\varepsilon,b_\varepsilon(0)}^{-1}(f_{\varepsilon,b_\varepsilon(0)}\tilde{g}_\varepsilon(t)f_{\varepsilon,b_\varepsilon(0)}^{-1})^\dagger,
$$
where $\cdot^\dagger$ is the adjoint with respect to $k$. $s$ is smooth with respect to any Sobolev norm and we have $s(0) = 0$ because $\tilde{g}_\varepsilon(0) = \Id_E$ so,
\begin{align*}
    -s(t) & \landau{t}{0} \e^{-s(t)} - \Id_E + \mathrm{o}(t)\\
    & \landau{t}{0} f_{\varepsilon,b_\varepsilon(0)}\tilde{g}_\varepsilon(t)f_{\varepsilon,b_\varepsilon(0)}^{-1}f_{\varepsilon,b_\varepsilon(0)}^{-1\dagger}\tilde{g}_\varepsilon(t)^\dagger f_{\varepsilon,b_\varepsilon(0)}^\dagger - \Id_E + \mathrm{o}(t)\\
    & \landau{t}{0} (f_{\varepsilon,b_\varepsilon(0)}\tilde{g}_\varepsilon'(0)f_{\varepsilon,b_\varepsilon(0)}^{-1} + f_{\varepsilon,b_\varepsilon(0)}^{-1\dagger}\tilde{g}_\varepsilon'(0)^\dagger f_{\varepsilon,b_\varepsilon(0)}^\dagger)t + \mathrm{o}(t)\\
    & \landau{t}{0} 2\Re_k(f_{\varepsilon,b_\varepsilon(0)}\tilde{g}_\varepsilon'(0)f_{\varepsilon,b_\varepsilon(0)}^{-1})t + \mathrm{o}(t),
\end{align*}
where $\Re_k$ is the Hermitian part with respect to $k$. We deduce that,
\begin{equation}\label{EQ:s'(0)}
    s'(0) = -2\Re_k(f_{\varepsilon,b_\varepsilon(0)}\tilde{g}_\varepsilon'(0)f_{\varepsilon,b_\varepsilon(0)}^{-1}).
\end{equation}
Now, recall that for all $t$, $\tilde{g}_\varepsilon(t) = \e^{\sigma(\varepsilon,b_\varepsilon(0))}g_\varepsilon(t)\e^{-\sigma(\varepsilon,b_\varepsilon(t))}$. Therefore, since $\norme{b_\varepsilon(0)} \landau{\varepsilon}{0} \mathrm{o}(1)$, we have,
\begin{align*}
    \tilde{g}_\varepsilon'(0) & \ =\ \e^{\sigma(\varepsilon,b_\varepsilon(0))}g_\varepsilon'(0)\e^{-\sigma(\varepsilon,b_\varepsilon(0))} - \e^{\sigma(\varepsilon,b_\varepsilon(0))}d\exp(-\sigma(\varepsilon,b_\varepsilon(0)))\frac{\partial}{\partial b}|_{b = b_\varepsilon(0)}\sigma(\varepsilon,b)b_\varepsilon'(0)\\
    & \landau{\varepsilon}{0} g_\varepsilon'(0) + \mathrm{o}(g_\varepsilon'(0)) + \mathrm{o}(b_\varepsilon'(0)) \textrm{ because $\sigma(0,0) = 0$ and $\frac{\partial}{\partial b}|_{b = 0}\sigma(0,b) = 0$,}\\
    & \landau{\varepsilon}{0} \i\nu_\varepsilon(b_\varepsilon(0)) + \mathrm{o}(\nu_\varepsilon(b_\varepsilon(0))) + \mathrm{o}(L_{b_\varepsilon(0)}\nu_\varepsilon(b_\varepsilon(0)))\\
    & \landau{\varepsilon}{0} \i\nu_\varepsilon(b_\varepsilon(0)) + \mathrm{o}(\nu_\varepsilon(b_\varepsilon(0))).
\end{align*}
We deduce from this equality and (\ref{EQ:s'(0)}) that,
\begin{align}
    f_{\varepsilon,b_\varepsilon(0)}^{-1}s'(0)f_{\varepsilon,b_\varepsilon(0)} & \ =\ -2f_{\varepsilon,b_\varepsilon(0)}^{-1}\Re_k(f_{\varepsilon,b_\varepsilon(0)}\tilde{g}_\varepsilon'(0)f_{\varepsilon,b_\varepsilon(0)}^{-1})f_{\varepsilon,b_\varepsilon(0)} \nonumber\\
    & \ =\ -2\Re_h(\tilde{g}_\varepsilon'(0)) \textrm{ because } k = f_{\varepsilon,b_\varepsilon(0)} \cdot h, \nonumber\\
    & \landau{\varepsilon}{0} -2\i\nu_\varepsilon(b_\varepsilon(0)) + \mathrm{o}(\nu_\varepsilon(b_\varepsilon(0))). \label{EQ:f-1 s'(0) f}
\end{align}

\vspace{1mm}
\noindent\textbf{Step 2 :} Relation between curvatures.
\vspace{1mm}

Since $f_{\varepsilon,b_\varepsilon(t)} \cdot h = \e^{-s(t)/2}u(t) \cdot k = \e^{-s(t)/2} \cdot k$, we have, by (\ref{EQ:Chasles}),
$$
\varphi_\varepsilon(b_\varepsilon(t)) - \varphi_\varepsilon(b_\varepsilon(0)) = M_{\varepsilon,\dbar_E}(h,\e^{-s(t)/2} \cdot k) - M_{\varepsilon,\dbar_E}(h,k) = M_{\varepsilon,\dbar_E}(k,\e^{-s(t)/2} \cdot k).
$$
Moreover, if we call $\partial_{E,k}$ the $(1,0)$ part of the Chern connection associated with $(\dbar_E,k)$ and we set $\partial_{\varepsilon,b_\varepsilon(0)} = f_{\varepsilon,b_\varepsilon(0)}^{-1} \circ \partial_{E,k} \circ f_{\varepsilon,b_\varepsilon(0)}$, we can compute that the connection $\nabla_{\varepsilon,b_\varepsilon(0)} = \partial_{\varepsilon,b_\varepsilon(0)} + \dbar_{\varepsilon,b_\varepsilon(0)}$ is unitary with respect to $h$. Thus, by uniqueness, $\nabla_{\varepsilon,b_\varepsilon(0)}$ is the Chern connection associated with $(\dbar_{\varepsilon,b_\varepsilon(0)},h)$. Let $\nabla_{E,k} = \partial_{E,k} + \dbar_E$ be the associated connection. Notice that
$$
\nabla_{E,k} = f_{\varepsilon,b_\varepsilon(0)} \circ \nabla_{\varepsilon,b_\varepsilon(0)} \circ f_{\varepsilon,b_\varepsilon(0)}^{-1}.
$$
These connections are conjugated. It is due to the fact that $f_{\varepsilon,b_\varepsilon(0)} : (E,\dbar_{\varepsilon,b_\varepsilon(0)},h) \rightarrow (E,\dbar_E,k)$ is both holomorphic and unitary.

Therefore,
$$
F_{\dbar_E,k} = \nabla_{E,k} \circ \nabla_{E,k} = f_{\varepsilon,b_\varepsilon(0)} \circ \nabla_{\varepsilon,b_\varepsilon(0)} \circ \nabla_{\varepsilon,b_\varepsilon(0)} \circ f_{\varepsilon,b_\varepsilon(0)}^{-1} = f_{\varepsilon,b_\varepsilon(0)}F_{\dbar_{\varepsilon,b_\varepsilon(0)},h}f_{\varepsilon,b_\varepsilon(0)}^{-1}.
$$
Thus,
\begin{equation}\label{EQ:Contractions des courbures conjuguées}
    \Lambda_\varepsilon\i F_{\dbar_E,k} - c_\varepsilon\Id_E = \i f_{\varepsilon,b_\varepsilon(0)}\nu_\varepsilon(b_\varepsilon(0))f_{\varepsilon,b_\varepsilon(0)}^{-1}.
\end{equation}

\vspace{1mm}
\noindent\textbf{Step 3 :} Conclusion.
\vspace{1mm}

Finally, by (\ref{EQ:Dérivée Donaldson}), when $\varepsilon \rightarrow 0$,
\begin{align*}
    \frac{\partial}{\partial t}|_{t = 0}\varphi_\varepsilon(b_\varepsilon(t)) & \ =\ \int_X \tr\left(s'(0)(\Lambda_\varepsilon\i F_{\dbar_E,k} - c_\varepsilon\Id_E)\right)\Vol_\varepsilon\\
    & \ =\ \int_X \tr\left(f_{\varepsilon,b_\varepsilon(0)}^{-1}s'(0)f_{\varepsilon,b_\varepsilon(0)}\i\nu_\varepsilon(b_\varepsilon(0))\right)\Vol_\varepsilon \textrm{ by (\ref{EQ:Contractions des courbures conjuguées}),}\\
    & \landau{\varepsilon}{0} \int_X \tr\left((-2\i\nu_\varepsilon(b_\varepsilon(0)) + \mathrm{o}(\i\nu_\varepsilon(b_\varepsilon(0))))\i\nu_\varepsilon(b_\varepsilon(0))\right)\Vol_\varepsilon \textrm{ by (\ref{EQ:f-1 s'(0) f}),}\\
    & \landau{\varepsilon}{0} -2\norme{\nu_\varepsilon(0)}_\varepsilon^2 + \mathrm{o}(\norme{\nu_\varepsilon(0)}_\varepsilon^2).
\end{align*}
\end{proof}

We are now ready to conclude this subsection.
\begin{proposition}\label{PRO:Limite b_varepsilon}
    If $\E$ is $[\Theta_\varepsilon]$-semi-stable, there exists a $b_{\infty,\varepsilon} \in B \cap \overline{\mathcal{O}}$ such that $\nu_\varepsilon(b_{\infty,\varepsilon}) = 0$. If moreover, $\E$ is $[\Theta_\varepsilon]$-stable, then $b_{\infty,\varepsilon} \in \mathcal{O}$. Moreover, $\norme{b_{\infty,\varepsilon}} = \mathrm{O}\!\left(\norme{\varepsilon}_{\mC^0} + \sqrt{\norme{[\varepsilon]}}\right)$.
\end{proposition}
\begin{proof}
Assume first of all that $\E$ is $[\Theta_\varepsilon]$-stable. Let for all $t$, $f_{\varepsilon,b_\varepsilon(t)} \cdot h = \e^{-s(t)/2} \cdot h$ for some $s(t)$ Hermitian with respect to $h$. We may assume that for all $t$,
$$
\int_X \tr(s(t))\Vol_0 = 0.
$$
Then, by the inequality (\ref{EQ:Simpson}) and Proposition \ref{PRO:phi décroissante}, for all $t$,
$$
\norme{s(t)}_\varepsilon \leq C_{1,\varepsilon} + C_{2,\varepsilon}M(h,\e^{-s(t)/2} \cdot h) = C_{1,\varepsilon} + C_{2,\varepsilon}\varphi_\varepsilon(b_\varepsilon(t)) \leq C_{1,\varepsilon} + C_{2,\varepsilon}\varphi_\varepsilon(b_\varepsilon(0)).
$$
It means that $t \mapsto s(t)$ is $L^2$ bounded. We have by definition,
$$
\e^{-s(t)} = f_{\varepsilon,b_\varepsilon(0)}\tilde{g}_\varepsilon(t)f_{\varepsilon,b_\varepsilon(0)}^{-1}(f_{\varepsilon,b_\varepsilon(0)}\tilde{g}_\varepsilon(t)f_{\varepsilon,b_\varepsilon(0)}^{-1})^*,
$$
thus $t \mapsto \tilde{g}_\varepsilon(t)$ and $t \mapsto \tilde{g}_\varepsilon(t)^{-1}$ are $L^2$ bounded. We also have,
$$
g_\varepsilon(t) = \e^{-\sigma(\varepsilon,b_\varepsilon(0))}\tilde{g}_\varepsilon(t)\e^{\sigma(\varepsilon,b_\varepsilon(t))},
$$
thus $t \mapsto g_\varepsilon(t)$ and $t \mapsto g_\varepsilon(t)^{-1}$ are bounded too (for any norm since they lie in a finite dimensional vector space). Therefore, we can find a sequence $t_m \rightarrow +\infty$ such that $g_\varepsilon(t_m) \rightarrow g_\varepsilon \in G$ is invertible and in particular,
$$
b_\varepsilon(t_m) \rightarrow b_{\infty,\varepsilon} = g_\varepsilon \cdot b_\varepsilon(0) \in \mathcal{O}.
$$
Let $l = \nu_\varepsilon(b_{\infty,\varepsilon}) \in \frak{k}$.

Since the norm of the moment map decreases by Proposition \ref{PRO:Flot défini}, $\norme{\nu_\varepsilon(b_\varepsilon(t))}_\varepsilon \rightarrow \norme{l}_\varepsilon$. By Proposition \ref{PRO:phi décroissante}, up to shrinking $U$,
$$
\frac{\partial}{\partial t}\varphi_\varepsilon(b_\varepsilon(t)) \landau{\varepsilon}{0} -2\norme{\nu_\varepsilon(b_\varepsilon(t))}_\varepsilon^2 + \mathrm{o}(\norme{\nu_\varepsilon(b_\varepsilon(t))}_\varepsilon^2) \leq -\norme{l}_\varepsilon^2 \textrm{ for $t$ large enough},
$$
and inequality (\ref{EQ:Simpson}) implies that $\varphi$ is bounded from below hence $l = 0$.

In the semi-stable case, we have by Theorem \ref{THE:Structures cônes} that $\varepsilon = \lim_{m \rightarrow +\infty} \varepsilon_m$ for some $\varepsilon_m$ such that $\E$ is $[\Theta_{\varepsilon_m}]$-stable. The $b_{\infty,\varepsilon_m}$ remain in a compact set included in $B$ so, up to extraction, we may build $b_{\infty,\varepsilon} \in B \cap \overline{\mathcal{O}}$ as $\lim_{m \rightarrow +\infty} b_{\infty,\varepsilon_m}$. The bound on the norm of $b_{\infty,\varepsilon}$ is given by Proposition \ref{PRO:Inégalité norme b}.
\end{proof}

\subsection{Convergence when $\varepsilon \rightarrow 0$ and continuity results}

A consequence of Proposition \ref{PRO:Limite b_varepsilon} is Introduction's Theorem \ref{THE:3}.
\begin{theorem}\label{THE:Famille de connections HYM}
    Let $(X,\Theta_0)$ be a compact balanced manifold and $\E$ a $[\Theta_0]$-semi-stable sufficiently smooth holomorphic vector bundle. Assume that $\mC_s(\E) \neq \emptyset$. There is a $\mC^0$ open neighbourhood $U$ of $0$ in the space of closed $(n - 1,n - 1)$ forms and a family $(\dbar_\varepsilon)_{\varepsilon \in U|\Theta_\varepsilon \in \mC_{ss}(\E)}$ of integrable Dolbeault operators on $E$ such that for all $\varepsilon$, $\dbar_\varepsilon$ is $\Theta_\varepsilon$-HYM and $(E,\dbar_\varepsilon) \cong \Gr_{[\Theta_\varepsilon]}(\E)$. Moreover,
    $$
    \norme{\dbar_\varepsilon - \dbar_0}_A = \mathrm{O}\!\left(\norme{\varepsilon}_{A'} + \sqrt{\norme{[\varepsilon]}}\right).
    $$
    In particular, $\dbar_\varepsilon \tend{\varepsilon}{0} \dbar_0$ for the $\mC^\infty$ topology on $U$ and on the space of Dolbeault operators.
\end{theorem}
\begin{proof}
Let for all $\varepsilon \in U$ such that $\E$ is $[\Theta_\varepsilon]$-semi-stable, $\dbar_\varepsilon = \dbar_0 + \tilde{\Phi}(\varepsilon,b_{\infty,\varepsilon})$ where $b_{\infty,\varepsilon}$ is given by Proposition \ref{PRO:Limite b_varepsilon}. The HYM condition is fulfilled and when $\Theta_\varepsilon$ is in the stable cone, $b_{\infty,\varepsilon} \in \mathcal{O}$ so,
$$
(E,\dbar_\varepsilon) \cong \E \cong \Gr_{[\Theta_\varepsilon]}(\E).
$$
When $\Theta_\varepsilon$ is only in the semi-stable cone, by \cite[Theorem 4]{Buchdahl_Schumacher}\footnote{In this paper, the results are proven in the Kähler case. However, in the third point of their conclusion and in the remark following \cite[Proposition 3.2]{Buchdahl_Schumacher}, the authors notice that the only obstruction for it to work on any compact complex manifold equipped with a Gauduchon metric is that in this case, the slope might not be of topological nature. In the balanced case, the slope is of topological nature since it only depends on the first Chern class of the bundle and not on the choice of its representative.},
$$
(E,\dbar_\varepsilon) \cong (E,\dbar_0 + \Phi(b_{\infty,\varepsilon})) \cong \Gr_{[\Theta_\varepsilon]}(\E).
$$

By smoothness of $\sigma$ given by Proposition \ref{PRO:Tranche déformée} and the bound on the norm of $b$ given by Proposition \ref{PRO:Limite b_varepsilon}, we have, when $\varepsilon \rightarrow 0$,
\begin{align*}
    \norme{\dbar_\varepsilon - \dbar_0}_A & = \norme{\e^{\sigma(\varepsilon,b_{\infty,\varepsilon})} \cdot (\dbar_0 + \Phi(b_{\infty,\varepsilon})) - \dbar_0}_A\\
    & = \norme{\e^{\sigma(\varepsilon,b_{\infty,\varepsilon})}\dbar_0(\e^{-\sigma(\varepsilon,b_{\infty,\varepsilon})}) + \e^{\sigma(\varepsilon,b_{\infty,\varepsilon})}\Phi(b_{\infty,\varepsilon})\e^{-\sigma(\varepsilon,b_{\infty,\varepsilon})}}_A \textrm{ by (\ref{EQ:Action de gauge}),}\\
    & = \mathrm{O}(\norme{\sigma(\varepsilon,b_{\infty,\varepsilon})}_{'\!A}) + \mathrm{O}(\norme{\Phi(b_{\infty,\varepsilon})}_A)\\
    & = \mathrm{O}(\norme{\varepsilon}_{A'} + \norme{b_{\infty,\varepsilon}}^2) + \mathrm{O}(\norme{b_{\infty,\varepsilon}})\\
    & = \mathrm{O}\!\left(\norme{\varepsilon}_{A'} + \norme{\varepsilon}_{\mC^0} + \sqrt{\norme{[\varepsilon]}}\right)\\
    & = \mathrm{O}\!\left(\norme{\varepsilon}_{A'} + \sqrt{\norme{[\varepsilon]}}\right).
\end{align*}
\end{proof}
\begin{remark}
Clearly, the same bound holds for $\norme{\partial_\varepsilon - \partial_0}_A$ and $\norme{\nabla_\varepsilon - \nabla_0}_A$, where $\nabla_\varepsilon = \partial_\varepsilon + \dbar_\varepsilon$ is the Chern connection associated to $(\dbar_\varepsilon,h)$.
\end{remark}

Let $\mathcal{H}(E)$ be the set of holomorphic structures on $E$ \textit{i.e.} the set of the integrable Dolbeault operators. We endow it with the $\mC^\infty$ topology. By \cite[Corollary 7.1.15]{Kobayashi}, the action of $\G$ on $\mathcal{H}(E)$ is proper so the space $\mathcal{H}(E)/\G$ is Hausdorff. We recall that $\underline{\mC}_{ss}(\E) \subset \mC_{ss}(\E)$ is the cone all balanced $\Theta$ such that $\E$ is $[\Theta]$-semi-stable and sufficiently smooth. We endow $\underline{\mC}_{ss}(\E) \subset \Omega^{n - 1,n - 1}(X,\R)$ with the $\mC^\infty$ topology too. We are now ready to prove Introduction's Theorem \ref{THE:4}.

\begin{theorem}\label{THE:Gamma continue}
    The function,
    $$
    \Gamma_\E : \fonction{\underline{\mC}_{ss}(\E)}{\mathcal{H}(E)/\G}{\Theta}{\{\dbar \in \mathcal{H}(E)|(E,\dbar) \cong \Gr_{[\Theta]}(\E) \textrm{ and $\dbar$ is $\Theta$-HYM}\}},
    $$
    is well-defined and continuous.
\end{theorem}
\begin{proof}
For all $\Theta \in \underline{\mC}_{ss}(\E)$, $\Gr_{[\Theta]}(\E)$ is a $[\Theta]$-polystable vector bundle so Proposition \ref{PRO:Kobayashi--Hitchin} tells us that the set,
$$
\{\dbar \in \mathcal{H}(E)|(\E,\dbar) \cong \Gr_{[\Theta]}(\E) \textrm{ and $\dbar$ is $\Theta$-HYM}\},
$$
is exactly one $\G$-orbit so $\Gamma_\E$ is well defined and indeed takes values in $\mathcal{H}(E)/\G$. Let $K \subset \mathcal{B}_X$ be a convex compact subset. It is enough to show continuity on the polyhedral cone $C = \underline{\mC}_{ss}(\E) \cap K$ for all such $K$ and we may assume that $C \neq \emptyset$. We have to consider three cases.

First case : $\mC_s(\E) \cap C \neq \emptyset$. In this case, continuity of $\Gamma$ at each point $\Theta_0$ is a consequence of Theorem \ref{THE:Famille de connections HYM}.

Second case : It exists $\Theta_0 \in C$ such that $F_{[\Theta_0]}$ is the whole cone $C$ and $\E$ is $[\Theta_0]$-polystable. In this case, let us write,
$$
\E = \bigoplus_{k = 1}^m \E_k,
$$
with the $\E_k$ $[\Theta_0]$-stable with the same $[\Theta_0]$-slope. For all $\Theta \in \underline{\mC}_{ss}(\E)$, each $\E_k$ is $[\Theta]$-semi-stable and have the same $[\Theta]$-slope as $\E$ because $[\Theta] \in F_{[\Theta_0]}$. It implies that,
$$
\Gr_{[\Theta]}(\E) = \bigoplus_{k = 1}^m \Gr_{[\Theta]}(\E_k).
$$
Moreover, any $\Theta$-HYM connection on $\Gr_{[\Theta]}(\E)$ makes this decomposition holomorphic and induces a $\Theta$-HYM connection on each $\Gr_{[\Theta]}(\E_k)$. Therefore, it makes sense to write that $\Gamma_\E = \sum_{k = 1}^m \Gamma_{\E_k}$. The fact that each $\Gamma_{\E_k}$ is continuous is the first case because for all $k$, $\Theta_0 \in \mC_s(\E_k) \cap C$.

Third case : General case. $C$ is assumed to be non-empty so there is a $\Theta_0$ such that $F_{[\Theta_0]} = C$. Thus, for all $\Theta \in \underline{\mC_{ss}}(\E)$, $F_{[\Theta]} \subset F_{[\Theta_0]}$. Therefore, a consequence of Lemma \ref{LEM:Gr(Gr)} is that,
$$
\Gamma_\E = \Gamma_{\Gr_{[\Theta_0]}(\E)},
$$
and $\Gr_{[\Theta_0]}(\E)$ is a $[\Theta_0]$-polystable vector bundle. We conclude thanks to the second case.
\end{proof}

\bibliographystyle{plainurl}

\end{document}